\documentclass[11pt]{amsart}
\usepackage{amsmath,amsfonts,amsthm,amssymb,amscd,hyperref}
\def\classification#1{\def\@class{#1}}
\classification{\null}

\DeclareFontFamily{OT1}{rsfs}{}
\DeclareFontShape{OT1}{rsfs}{n}{it}{<-> rsfs10}{}
\DeclareMathAlphabet{\mathscr}{OT1}{rsfs}{n}{it}

\DeclareMathOperator{\supp}{supp}

\DeclareMathOperator{\diam}{diam}
\DeclareMathOperator{\leng}{length}

\DeclareMathOperator{\PSL}{PSL}

\DeclareMathOperator{\Vol}{Vol}

\DeclareMathOperator{\Sym}{Sym}
\DeclareMathOperator{\Prob}{Prob}
\DeclareMathOperator{\Alt}{Alt}

\DeclareMathOperator{\Var}{Var}

\newtheorem{prop}{Proposition}[section]
\newtheorem{thm}[prop]{Theorem}
\newtheorem*{main}{Main Theorem}

\newtheorem{conj}{Conjecture}

\newtheorem{cor}[prop]{Corollary}
\newtheorem{lem}[prop]{Lemma}

\theoremstyle{remark}

\numberwithin{equation}{section}

\begin{document}
\title{On the diameter of permutation groups}
\author{Harald A. Helfgott}
\address{Harald A. Helfgott, 
\'Ecole Normale Sup\'erieure, D\'epartement de Math\'ematiques, 45 rue d'Ulm, F-75230 Paris, France}
\email{harald.helfgott@ens.fr}
\author{\'{A}kos Seress}\thanks{\'{A}kos Seress passed away on
February 13, 2013, after the paper's acceptance.}
\address{\'{A}kos Seress,
Centre for the Mathematics of
Symmetry and Computation,
The University of Western Australia,
Crawley, WA 6009 Australia,
and
Department of Mathematics,
The Ohio State University,
Columbus, OH 43210, USA}

\begin{abstract}
Given a finite group $G$ and a set $A$ of generators, the diameter
$\diam(\Gamma(G,A))$ of the Cayley graph $\Gamma(G,A)$ is the smallest
$\ell$ such that every element of $G$ can be expressed as a word of length
at most $\ell$ in $A \cup A^{-1}$. We are concerned with bounding
$\diam(G):= \max_A\diam(\Gamma(G,A))$.

It has long been conjectured that the diameter of the symmetric group of
degree $n$ is polynomially bounded in $n$, but the best previously known
upper bound was exponential in $\sqrt{n \log n}$. We give a
quasipolynomial upper bound, namely,
\[\diam(G) = \exp\left(O((\log n)^4 \log\log n)\right) = \exp\left((\log \log |G|)^{O(1)}\right)\]
for $G = \Sym(n)$ or $G = \Alt(n)$, where the implied constants are absolute. 
This addresses a key open case of Babai's conjecture
on diameters of simple groups. By a result of Babai and Seress (1992), 
our bound also implies a quasipolynomial upper bound  on the diameter of all 
transitive permutation groups of degree $n$.

\end{abstract}

\maketitle

\section{Introduction}
\subsection{Groups and their diameters}
Let $A$ be a set of generators for a group $G$. The (undirected)
{\em Cayley graph} $\Gamma(G,A)$ is the graph whose set of vertices is $V=G$ and whose
set of edges is $E = \{ \{ g,ga \} : g\in G, a\in A\}$. The {\em diameter} $\diam(\Gamma)$ of a graph $\Gamma(V,E)$ is defined by
\begin{equation}
\label{eq:dist}
\diam(\Gamma) = \max_{v_1,v_2\in V} \mathop{\min_{\text{$P$ a path}}}_{\text{from $v_1$ to $v_2$}}
\leng(P).
\end{equation}
In particular, the diameter of a Cayley graph  $\Gamma(G,A)$ is the maximum, for $g\in G$,
of the length $\ell$ of the shortest expression 
$g = a_1^{\varepsilon_1} a_2^{\varepsilon_2}\cdots a_\ell^{\varepsilon_\ell}$ with $a_i \in A$ and $\varepsilon_i \in \{ -1,1\}$ for each
$i = 1,\dotsc,\ell$.  We may define the diameter $\diam(G)$ of a finite group to be the maximal 
diameter of the Cayley graphs $\Gamma(G,A)$ for all generating sets $A$ of $G$.

Much recent work on group diameters has been motivated by the following
conjecture:
\begin{conj}
\label{babaiconj}
$($Babai, published as \cite[Conj. 1.7]{BS92}$)$
For all finite simple groups $G$, 
\[
\diam(G) \le (\log |G|)^{O(1)},
\]
where the implied constant is absolute.
\end{conj}
Here and henceforth, $|S|$ denotes the number of elements of a set $S$.

The first class of finite simple groups for which
Conj.~\ref{babaiconj} was established was 
$\PSL_2(\mathbb{Z}/p\mathbb{Z})$ with $p$ prime, by Helfgott
\cite{Hel08}. 
The paper \cite{Hel08} initiated a period of intense
 activity 
\cite{BGSU2}, \cite{MR2415383}, \cite{Din}, \cite{MR2587341}, \cite{HeSL3}, 
\cite{GH1}, \cite{Var}, \cite{BGSup}, \cite{PS}, \cite{BGT}, \cite{GH2}, \cite{SGV}\footnote{This list is not meant to be exhaustive.} 
on the diameter problem and the related problem of expansion properties of Cayley graphs. 

As far as work in this vein on the diameter of
 finite simple groups is concerned, the best
results to date are those of Pyber, Szab\'o~\cite{PS} and Breuillard, Green, Tao~\cite{BGT}. Their wide-ranging
generalisation covers all simple groups of Lie type, but (just like \cite{GH1})
the diameter estimates retain a strong dependence on the rank; thus, they prove Conj.~\ref{babaiconj} only for groups of bounded rank. The problem for the alternating groups remained wide open.\footnote{See, e.g.,
I. Pak's remarks (made already before \cite{PS}, \cite{BGT}) on the
relative difficulty of the work remaining to do in the linear case 
(to be finished ``in the next 10 years'') 
and of the problem on $\Alt(n)$, for which 
there was ``much less hope'' \cite{Paktalk}.}

These two issues are arguably related: product theorems
(of the 
type $|A \cdot A \cdot A|\gg |A|^{1+\delta}$ familiar since \cite{Hel08})
are false both in the unbounded-rank case and in the case of
 alternating groups, and the counterexamples described in both situations
 in \cite{MR2898694}, \cite{PS} are based on similar principles.

In the present paper we address the case of alternating
(and symmetric) groups.   We expect that some of the combinatorial
difficulties we overcome will also arise in the context of linear groups
 of large rank.

For $G=\Alt(n)$, Conj.~\ref{babaiconj} stipulates that $\diam(\Alt(n)) = n^{O(1)}$;
\cite{BS92} refers to this special case of Conj.~\ref{babaiconj} as a ``folklore'' conjecture. Indeed, this has long been a problem of 
interest in computer science (see \cite{KMS84}, \cite{McK}, \cite{BHKLS}, \cite{BBS04}, \cite{BH}). On a more playful level,
bounds on the diameter of permutation groups are relevant to every permutation
puzzle (e.g., Rubik's cube).

The best previously known upper bound on $\diam(G)$ for $G=\Alt(n)$ or
$G = \Sym(n)$ was more than two decades old: 
\begin{equation}
\label{trangeneral}
\diam(G) \leq  \exp((1+o(1)) \sqrt{n \log n}) = \exp((1+o(1))
  \sqrt{\log |G|}),
\end{equation}
due to Babai and Seress \cite{BS88}. 
(We write $\exp(x)$ for $e^x$.) 

\subsection{Statement of results}
Recall that a function $f(n)$ is called {\em quasipolynomial} if $\log (f(n))$
is a polynomial function of $\log n$. Our main result establishes a 
quasipolynomial 
upper bound for $\diam(\Alt(n))$ and $\diam(\Sym(n))$. 

\begin{main}
Let $G = \Sym(n)$ or $\Alt(n)$. Then
\[\diam(G)\leq  \exp\left(O((\log n)^4 \log\log n)\right), \]
where the implied constant is absolute.
\end{main}


The quasipolynomial bound extends to a much broader class
of permutation groups.  Recall that a permutation group $G$ acting
on a set $\Omega$ is called \emph{transitive}
if 
\[\forall \alpha,\beta\in \Omega\;\;\;\;\; \exists g\in G \text{\;such that\;} g \text{\;takes\;} \alpha\; \text{to}\; \beta.\]
The size $|\Omega|$ of the permutation domain is called the {\em degree} of 
 $G$.

Kornhauser et al. \cite{KMS84} and McKenzie \cite{McK} raised the question of
what classes of permutation groups may have polynomial diameter bound in their
degree. 
A weaker, quasipolynomial bound for all transitive groups was formally conjectured in \cite{BS92}:
\begin{conj}
\label{bsconj}
$($\cite[Conj. 1.6]{BS92}$)$
If $G$ is a transitive permutation group of
degree $n$ then $\diam(G) \le \exp((\log n)^{O(1)})$.
\end{conj}

Babai and Seress~\cite{BS92} linked Conj.~\ref{bsconj} to the
diameter of alternating groups:

\begin{thm}
\label{trandiameter}
$($\cite[Thm. 1.4]{BS92}$)$
If $G$ is a transitive permutation group
of degree $n$ then 
$$\diam(G) \le \exp\left(O(\log n)^3\right) \diam\left(\Alt(k)\right),$$
where $\Alt(k)$ is the largest alternating composition factor of $G$.
\end{thm}

Combining our Main Theorem with Thm.~\ref{trandiameter}, we  immediately obtain

\begin{cor}
\label{bsconjtrue}
Conjecture~$\ref{bsconj}$ is true; indeed the diameter of
any transitive permutation group $G$ of degree $n$ is 
\[    \diam(G) \le \exp \left(O((\log n)^4 \log\log n) \right).  \]
\end{cor}

We note that Thm.~\ref{trandiameter} is not only used to prove
Cor.~\ref{bsconjtrue} -- it also comes into play as an inductive tool in the proof of the Main Theorem (see Lemma \ref{lem:cases3a}). Since Thm.~\ref{trandiameter}
relies on the Classification of Finite Simple Groups, so does the Main
Theorem.

It is well-known that, for any finite group $G$ and any set $A$ of generators
of $G$, the eigenvalues $\lambda_0\geq \lambda_1 \geq 
\lambda_2\geq \dotsc$
of the adjacency matrix of $\Gamma(G,A)$ satisfy
\begin{equation}\label{eq:gapbo}
\lambda_0 - \lambda_1 \geq \frac{1}{\diam(\Gamma(G,A))^2}.\end{equation}
(See \cite[Cor.\ 1]{MR1245303} or the references 
\cite{Aldous}, \cite{Babtech}, \cite{Gangthe}, \cite{Mohtech} therein.)
Because of (\ref{eq:gapbo}), we obtain immediately that
\[\lambda_0 - \lambda_1 \geq \exp(- O((\log n)^4 \log\log n)),\]
with consequences on expansion and the mixing rate (see, e.g., 
\cite{MR1395866}, \cite{HLW}).

Finally, the Main Theorem and Cor.~\ref{bsconjtrue} extend to directed graphs. 
Given $G=\langle A \rangle$, the {\em directed Cayley graph}
$\vec{\Gamma}(G,A)$ is the graph with vertex set $G$ and edge set $\{ (g,ga)
: g \in G, a \in A \}$. The diameter of $\vec{\Gamma}(G,A)$ is defined by
\eqref{eq:dist}, where ``path'' should be read as ``directed path''; 
$\overrightarrow{\diam}(G)$ is the maximum of
$\diam(\vec{\Gamma}(G,A))$ taken as $A$ varies over all generating sets $A$
of $G$. Thanks to Babai's bound
$\overrightarrow{\diam}(G) = O\left(\diam(G)
\cdot (\log |G|)^2
\right)$ \cite[Cor.\ 2.3]{MR2368881}, valid for all groups $G$,
 we obtain immediately
from Cor.~\ref{bsconjtrue} that
\begin{cor}
\label{directed main}
Let $G$ be a transitive group on $n$ elements. Then 
\[\overrightarrow{\diam}(G) \le \exp (O((\log n)^4 \log\log n)).\]
\end{cor}




\subsection{General approach}\label{sec:genapp}
An analogy underlies recent work on growth in groups: much\footnote{Or
at least results on subgroups that rely on {\em grosso modo}
quantitative arguments. (Crucially, the orbit-stabilizer theorem carries
over (Lem.\ \ref{lem:orbsta}); Sylow theory, which is quantitative but relies on
(necessarily delicate) congruences, does not.) 
As \cite[Lem.\ 2.1]{BBS04} (in retrospect) and Prop.\ \ref{prop:babai}
 in the present work make clear, probabilistic arguments
in combinatorics can also carry over, provided that the desired probability
distribution on a set can be approximated quickly by the action of a random
walk.} of basic
group theory carries over when, instead of subgroups, we study sets
that grow slowly ($|A \cdot A \cdot A|\leq |A|^{1+\varepsilon}$). This realisation is clearer in
\cite{HeSL3} than in \cite{Hel08}, and has become current since then.
(The term ``approximate group'' \cite{MR2501249} actually first arose 
in a different context, namely, the
generalisation of some arguments in classical additive combinatorics to
the non-abelian case. (See also \cite[\S 2.3]{Hel08}, 
\cite[Lem.\ 4.2]{MR2155059}.) The analogy between subgroups and slowly
growing
sets was also explored in a model-theoretic setting in later work by Hrushovski 
\cite{Hrushovski}.) 

This analogy is more important than
whether one works with 
approximate subgroups in Helfgott's sense
($|A \cdot A \cdot A| \leq |A|^{1+\varepsilon}$, or
more generally $|A \cdot A\cdot A|\leq f(|A|)$ for some specified $f$) or Tao's sense 
\cite[Def.\ 3.7]{MR2501249}; 
the two definitions
are essentially equivalent, and we will actually work with neither.
We could phrase part of our argument in terms of statements of the form
$|A^k|\leq |A|^{1+\varepsilon}$, but $k$ would sometimes be larger than $n$;
applying the tripling lemma (\cite{MR810596}, 
\cite[Lem.\ 2.2]{Hel08}, \cite[Lem.\ 3.4]{MR2501249}) 
to such statements would weaken them fatally.

There is another issue worth emphasising: the study of growth needs to be
relative. We should not think simply in terms of a group acting on
itself by multiplication -- even if, in the last analysis, this is the
only operation available to us. Rather, growth statements often need to
be thought of in terms of the action of a group $G$ on a set $X$, and the
effect of this action on subsets $A\subseteq G$, $B\subseteq X$. (Here $X$ may
or may not be endowed with a structure of its own.) This was already
clear in \cite[Prop.\ 3.1]{HeSL3} and \cite{GH2}, and is crucial here: a key step will involve
the action of a normaliser $N_{G}(H)$ on a subgroup $H\leq G$ by conjugation.

\subsection{Relation to previous work}\label{subs:wyorel}
Our debt to previous work on permutation groups is manifold. It is
worthwhile to point out that some of our main techniques are adaptations
to sets of
classification-free arguments\footnote{Cf.\ the role of \cite{LP} (esp. Thm.\ 4.2, Thm.\ 6.2), which, in order
to provide alternatives to the Classification of Finite Simple Groups, did (both
more and less
generally) for subgroups what \cite[\S 5]{HeSL3} did for sets, and was
later translated back to sets for use in \cite{BGT}.
} on the properties of subgroups of
$\Sym(n)$ by Babai \cite{Bab82}, Pyber \cite{Pyb93}, Bochert \cite{Boch89}, and Liebeck \cite{MR703984}. 
Of particular importance is Babai and Pyber's work on the order of $2$-transitive
groups \cite{Bab82}, \cite{Pyb93}. 

We shall also utilise existing diameter bounds. Besides Thm~\ref{trandiameter}, we shall use the main idea from \cite{BS88} (see Lemma~\ref{lem:small support}) and the following theorem by
Babai, Beals, and Seress. For a permutation $g$ of a set $\Omega$, the 
{\em support} $\supp(g)$ is the subset of elements of $\Omega$ that are displaced by $g$.

\begin{thm}
\label{bbssmallsupport}
$($\cite{BBS04}$)$
For every $\varepsilon < 1/3$ there exists $K(\varepsilon)$ such that, 
if $G=\Alt(n)$ or $\Sym(n)$ and $A$ is a set of generators of $G$ containing
an element
$x \in A$ with $1 < |\supp(x)| \le \varepsilon n$, then
$$\diam(\Gamma(G,A)) \le K(\varepsilon)n^8.$$
\end{thm}
We will use this theorem repeatedly in \S \ref{sec:proof of main}. As we
shall make clear in \S \ref{sec:random walks},
we also apply -- crucially -- one of the main methods involved in the proof of Thm.~\ref{bbssmallsupport},
namely, the use of short random walks to mimic a uniform distribution.

We note that until recently Theorem \ref{bbssmallsupport} gave the largest known explicit class of
Cayley graphs of $\Sym(n)$ or $\Alt(n)$ that has polynomially bounded
diameter. In late 2010, partly based on ideas from \cite{BBS04}, Bamberg et al. \cite{7author} proved that if a set of generators 
of $\Sym(n)$ or $\Alt(n)$ contains an element of support size 
at most $0.63n$ 
then the diameter of the Cayley graph is bounded by a polynomial of $n$. 

\subsection{Outline}\label{subs:outl}

Let us begin {\em in medias res}, focusing on a crucial moment at
which growth is achieved. Classical reasons aside, this will allow us
to emphasize the link to \cite{Hel08}, \cite{HeSL3}, \cite{BGT},
\cite{PS} and \cite{GH2}, while repeating one of the main motifs:
growth results from the action of a group on a set, often, as is the
case here, by conjugation.

The setup for the crucial step will involve
a set $A\subset \Sym([n])$ with $A = A^{-1}$ and a fairly large
set $\Sigma\subset [n]$ ($[n] := \{1,2,\dotsc,n\}$) such that the pointwise
stabilizer\footnote{Defined as in (\ref{eq:tronk}). The notation here
follows Dixon and Mortimer
\cite{DM} and Seress \cite{MR1970241} rather than Wielandt
\cite{MR0183775}.
Wielandt writes $A_\Sigma$ for the pointwise stabilizer, which we
denote by
$A_{(\Sigma)}$; we write $A_\Sigma$ for the setwise stabilizer.}
$A_{(\Sigma)}$ generates a group $\langle A_{(\Sigma)}\rangle$ with
a large orbit $\Gamma\subset [n]\setminus \Sigma$. (Say, for concreteness, that $|\Sigma|\geq (\log n)^2$ and $|\Gamma|>0.95n$.)
The setwise stabilizer $\langle A_\Sigma\rangle$ acts on the pointwise
stabilizer $\langle A_{(\Sigma)} \rangle$ by conjugation.

We can assume that $\langle A_{(\Sigma)}\rangle$ acts as the
alternating or symmetric group on $\Gamma$,
as otherwise we are done by a different argument (called {\em descent} in
\S \ref{sec:proof of main}; we will discuss it later). It follows that we can
find a set $S$ of at most six elements of $(A_{(\Sigma)})^\ell$, 
$\ell$ fairly small, such that $\langle S\rangle$ is
doubly transitive on $\Gamma$. (This implication is far from trivial; we prove a 
general result of this kind (Cor. \ref{cor0.5}) showing that, if a set
$A'$ generates $\Sym(\lbrack m\rbrack)$ or $\Alt(\lbrack m\rbrack)$,
then there is a small set $S\subset (A')^\ell$, $\ell$ fairly small, such that
$\langle S\rangle$ is $k$-transitive.)
 
Consider the action of the elements of 
$A_{\Sigma}$ on the elements of $S$ by conjugation. By an orbit-stabilizer
principle, either (a) an element $g\ne e$ of $A_\Sigma$ fixes (i.e.,
commutes with) every element of $S$, or (b) the orbit $\{g s g^{-1} : 
g\in A_\Sigma\}$ of some $s\in S$ is of size $\geq |A_\Sigma|^{1/6}$.
In case (a), since $\langle S\rangle$ is doubly transitive, $g$ fixes
every point of $\Gamma$. We have thus constructed a non-identity element
$g\in A$ with small support, and are done by Thm. 
\ref{bbssmallsupport}. In case (b), we have constructed many 
($\geq |A_\Sigma|^{1/6}$) distinct elements $g s g^{-1}$  in the pointwise
stabilizer $(A^3)_{(\Sigma)}$. This is what we call {\em creation} in
\S \ref{sec:proof of main}.

The questions are now -- how do we get to the point at which we began
our narrative? And how do we use the conclusion we have just shown, namely,
the creation of many elements in the pointwise stabilizer?

Let us start with the first question.
For the conclusion to be strong, $A_\Sigma$ should be large -- for instance,
large in 
comparison to $A_{(\Sigma)}$ or $(A^2)_{(\Sigma)}$. Now, 
$A_\Sigma$ can be much larger than $(A^2)_{(\Sigma)}$ only if $A$ occupies a large
number $R$ of
 cosets of $\Sym(\lbrack n\rbrack)_{(\Sigma)}$
 in $\Sym(\lbrack n\rbrack)$.  
(By pigeonhole, $|(A^2)_{(\Sigma)}|\geq |A|/R$.) Our aim will be to 
find a large $\Sigma$ such that $R$ is larger than $(d n)^{|\Sigma|}$,
where $d>1/2$ is a constant. 

This is also an intermediate aim in \cite{Pyb93} (which treats subgroups,
not sets). Much as there, we use this as follows: $R$ is larger
than $(d n)^{|\Sigma|}$, and so $A A^{-1}$ intersects at least 
$d^{|\Sigma|} |\Sigma|!$ cosets of $(\Sym(\lbrack n\rbrack))_{(\Sigma)}$
within $(\Sym(\lbrack n\rbrack))_{\Sigma}$ (by pigeonhole); this means
that the projection (by restriction) of $(A A^{-1})_\Sigma$ to $\Sym(\Sigma)$ 
has size at least $d^{|\Sigma|} |\Sigma|!$. At this point Pyber
uses the fact (due to Liebeck \cite{MR703984} and based on Bochert 
\cite{Boch89}) that, if a subgroup $H$ of $\Sym(\Sigma)$ is
of size at least $s=d^{|\Sigma|} |\Sigma|!$, where $d>1/2$, then there
must be a large orbit $\Delta\subset \Sigma$ of $H$ such that the
restriction of $H$ to $\Delta$ equals $\Alt(\Delta)$
or $\Sym(\Delta)$. We will show
(Prop. \ref{prop:liebest}) that,
even if $H\subset \Sym(\Sigma)$ is just a set, not a subgroup, the assumption
that $H$ is of size at least $s$ implies that the restriction of 
$H^\ell$ to $\Delta$ equals all of $\Alt(\Delta)$
or $\Sym(\Delta)$, where $\ell$ is relatively small. (This works because
the proof of Bochert's nineteenth-century result is algorithmic.) 
The fact that we obtain all of $\Alt(\Delta)$ or $\Sym(\Delta)$ is 
particularly important for what we called a ``descent argument'' (as in
``infinite descent'') in the above.

Now, as we said, we must find a large $\Sigma$ such that 
$A$ (or $A^{\ell'}$, $\ell'$ moderate) occupies a large number of cosets 
of $\Sym(\lbrack n\rbrack)_{(\Sigma)}$,
i.e., sends $(\Sigma)$ to many different tuples. Pyber shows this
(for $A$ a subgroup) by
constructing $\Sigma=\{\alpha_1,\alpha_2,\dotsc,\alpha_m\}$ so that
\begin{equation}\label{eq:janacek}
|\alpha_i^{A_{(\alpha_1,\dotsc,\alpha_{i-1})}}|\geq d n
\end{equation}
for every $1\leq i\leq m$. (The use of stabilizer chains 
$A>A_{(\alpha_1)}>A_{(\alpha_1,\alpha_2)}>\dotsc$ goes back to the algorithmic
work of Sims \cite{MR0257203}, \cite{Sim71}, as does
the use of the size of the orbits in (\ref{eq:janacek}); see
\cite[\S 4.1]{MR1970241}.) This step also works when $A$ is a subset
(Lemma \ref{lem:hog}). The difficult part, of course, is to show that
elements $\alpha_1,\alpha_2,\dotsc,\alpha_m$ satisfying (\ref{eq:janacek}) 
exist.

Here \cite{Pyb93} uses Babai's splitting lemma \cite{Bab82}, which states
that, if $H<\Sym(\lbrack n\rbrack)$
is a doubly transitive permutation group
 and $\Sigma\subset \lbrack n\rbrack$ is such that 
$H_{(\Sigma)}$ has no orbits of size $>(1-\epsilon) n$,
 then there is a set $\Sigma'\subset
\lbrack n \rbrack$ with $|\Sigma'|\ll_\epsilon (\log n) |\Sigma|$ such that 
$H_{(\Sigma')}$ consists only of the identity. In fact,
$\Sigma' = \Sigma^S = \{x^S:x\in \Sigma, s\in S\}$, where $S$ is
a subset of $H$ of size $|S|\ll \log n$. Babai constructs $S$ by choosing
$O(\log n)$ elements randomly from $H$ with the uniform distribution. 
A random element of $H$ takes a pair $(x,y)$ of distinct elements
of $\lbrack n\rbrack$ to any other such pair $(x',y')$ with the same
probability ($(n(n-1)/2)^{-1}$) no matter what $(x',y')$ is.
Now, given any distinct $x,y\in \lbrack n\rbrack$,
it is almost certain that they will be taken to elements $x^g$, $y^g$ 
of different orbits of $H_{(\Sigma)}$ by {\em some} $g\in S
\subset H$, simply because a positive proportion of all pairs $(x',y')$
lie in different orbits (by the fact that there is no orbit of size
$>(1-\epsilon) n$). Then, $x$ and $y$ belong to 
different orbits of $g H_\Sigma g^{-1} = H_{\Sigma^{g^{-1}}}$, and thus to 
different orbits of $H_{\Sigma^S}$. Summing probabilities
over all $x$ and $y$, we obtain that, with positive probability, every
two distinct $x,y\in \lbrack n\rbrack$ belong to different orbits
of $H_{\Sigma^S}$. This implies that $H_{\Sigma^S}$ is trivial.

We adapt this entire argument so as to hold
for a set $A\subset \Sym(\lbrack n\rbrack)$ instead of a subgroup
$H<\Sym(\lbrack n\rbrack)$; as usual, sometimes $H$ is replaced by $A$
and sometimes by $A A^{-1}$ or $A^\ell$, where $\ell$ is moderate
($\ell\ll n^{O(1)}$). The key here is that the outcome of a random
walk of moderate length takes a pair $(x,y)$ to any other pair $(x',y')$
with almost uniform probability. 

We apply the resulting generalization of the splitting lemma
(Prop.\ \ref{prop:babai}) and point out that $(A A^{-1})_{(\Sigma')} = \{e\}$
implies $|\Sigma'|\gg \log_n |A|$ (by pigeonhole) and so 
$|\Sigma|\gg (\log |A|)/(\log n)^2$.
 In other words, we are guaranteed to
be able to construct a stabilizer chain with long orbits as in 
(\ref{eq:janacek}) (for any $d<1$) until $m$ gets to size
proportional to $(\log |A|)/(\log n)^2$. We call this the {\em
organizing} step.

Now that we have the stabilizer chain, and thus the proper setup
for the creation step, how do we use the outcome of the creation
step? In \cite{Hel08} and the work that followed, the main intermediate 
result stated that a generating set $A$ always grew in size 
($|A^3|\geq |A|^{1+\delta}$ 
\cite[Key Proposition]{Hel08}); to 
prove that the diameter $\Gamma(G,A)$ was small, one just had to 
apply this key proposition over and over 
($|A^3|\geq |A|^{1+\delta}$, $|A^9|\geq |A^3|^{1+\delta}
\geq |A|^{(1+\delta)^2}$, \dots). Here we will
also prove our diameter bound by iteration; however, the quantity
whose growth we will keep track of during iteration will not be the size
of $A^\ell$, but rather the length of the sequence $\alpha_1,\alpha_2,\dotsc$
we have constructed satisfying (\ref{eq:janacek}) (for $A^\ell$ instead of
$A$).

The iteration is conducted as follows. We actually construct the first
$(\log n)^2$ elements of $\alpha_1,\alpha_2,\dotsc$ by brute force,
by raising $A$ to an $n^{O((\log n)^2)}$th power.  (This works by Lemma 
\ref{lem:apeman1}.) Now we get to the main step that gets repeated
(Prop.\ \ref{growth}): given a sequence $\alpha_1,\dotsc,\alpha_m$
satisfying (\ref{eq:janacek}) (for $A^\ell$ instead of $A$), 
we use the creation step to construct at least
$(m!)^{1/6}$ elements of $(A^{\ell'})_{(\alpha_1,\dotsc,\alpha_{m})}$,
where $\ell' \leq n^{O(\log n)} \ell$; then we use the organizing step
to construct new elements $\alpha_{m+1},\dotsc,\alpha_{m'}$
($m'\geq m + c m (\log m)/(\log n)^2$) so that (\ref{eq:janacek}) is
satisfied for all $i=1,2,\dotsc,m'$ (with $A^{\ell'}$ instead of $A^\ell$).
(We actually repeat the organizing step several times after each creation
step; this helps us save a $\log$ in the final exponent.) Repeating this,
we keep on lengthening the sequence $\alpha_1,\alpha_2,\dotsc$ until it
gets to be of length almost $n$, and then we are done easily.

\begin{center}
* * *
\end{center}

Needless to say, in the above outline, we have left out details that will
be treated in full in the body of the text. Let us discuss one more thing
now -- namely, what we have called the {\em descent} step. We reach
it when we have constructed a set $\Sigma = \{\alpha_1,\alpha_2,\dotsc,
\alpha_m\}$ such that (a) the restriction of $A_{\Sigma}$ to $\Sigma$
acts as $\Alt(\Delta)$ or $\Sym(\Delta)$ on a large subset $\Delta\subset
\Sigma$, (b) the restriction of $\langle A_{\Sigma}\rangle$ 
to $\lbrack n\rbrack\setminus \Sigma$ does not act like $\Alt$ or $\Sym$
on any subset of $\lbrack n\rbrack\setminus \Sigma$ larger than
$0.95 n$ (say). 

Now we can use Thm. \ref{trandiameter} (Babai-Seress), and obtain from (b) 
that the diameter of $\langle A_{\Sigma}\rangle$ is bounded in terms
of the diameter of $\Alt(k)$, $k =\lbrack 0.95 n\rbrack$. (It is here,
and only here, that the Classification Theorem is needed, since
Thm. \ref{trandiameter} is based on it.) Now we can use, inductively,
our own main theorem on the diameter of $\Alt(n)$, with $k$ instead of $n$.
This gives a bound on the diameter of $\langle A_{\Sigma}\rangle$.
At this point we use Lemma \ref{lem:small support} (which is
\cite[Lemma 3]{MR894827}; see also \cite{BLS87}). This shows that (a) 
implies that $\langle A_{\Sigma}\rangle$ contains
a non-identity element $g$ of small support. We can now apply
Thm.\ \ref{bbssmallsupport} (Babai-Beals-Seress) to bound the diameter
of our group $G=\Alt(n)$ or $G=\Sym(n)$ with respect to $A$. Note that
\cite[Lemma 3]{MR894827} would be prohibitively expensive if used
as a constructive result; here we are using it to show the {\em existence}
of an element, which we know can be constructed as a relatively short
word thanks to the bound on the diameter of $\langle A_{\Sigma}\rangle$
we obtained through Thm. \ref{trandiameter}.

\subsection{Acknowledgements}
We are deeply grateful to both Pablo Spiga and Nick Gill for stimulating
discussions and for their constant help.
Gordon Royle organised the first author's visit to Australia;
if it were not for him, our collaboration might not have happened. Thanks are
also due to L\'aszl\'o Babai, Martin Kassabov, Igor Pak, Peter Sarnak and 
Andrzej \.{Z}uk for their advice. Detailed comments by two anonymous
referees have certainly helped improve the paper.

\'Akos Seress was supported in part by the NSF and by ARC Grant DP1096525.
Travel was supported in part by H. A. Helfgott's Philip Leverhulme prize.
We benefited from the kind hospitality of the University of Western Australia
and the \'Ecole Normale Sup\'erieure during our visits to each other's
institutions.

\section{Notation}
\label{sec:notation}
We write $[n]=\{ 1,2,\ldots, n \}$. 
For a set $\Omega$, $\Sym(\Omega)$ and $\Alt(\Omega)$ are the 
symmetric and alternating groups acting on $\Omega$. As is customary,
we often write $\Alt(n)$ and $\Sym(n)$ for $\Alt(\lbrack n\rbrack)$ and
$\Sym(\lbrack n\rbrack)$ - particularly when we are thinking of these
groups as abstract groups as opposed to their actions.

We write $H \le G$ to mean that $H$ is a subgroup of $G$ and 
$H \lhd G$ to mean that
 $H$ is a normal subgroup. We say that a group $S$ is a {\em section} of a group $G$ 
if there exist subgroups $H$ and $K$ of $G$ with $K \lhd H$ and $H/K\cong S$.
We denote the identity element of a group by $e$.

Let $A$ be a subset of a group $G$.
We write $A^{-1} = \{a^{-1} : a \in A\}$,
$A^k = \{a_1 a_2 \cdots a_k : a_1,\dotsc,a_k\in A\}$. In 
\cite{Hel08}, \cite{HeSL3}, the first author wrote $A_\ell$ to mean
$(A\cup A^{-1} \cup \{e\})^{\ell}$; this does not seem to have 
become standard, and would
also not do here due to the potential confusion with alternating groups.
(Recall that $A_n$ is in common usage as a synonym for $\Alt(n)$.)
We will often include $A=A^{-1}$, $e\in A$ explicitly in our assumptions
so as to simplify notation. 
A set $A$ with $A = A^{-1}$ is said to be {\em symmetric}.

We write $|A|$ for the number of elements of a set $A$.
(All of our sets and groups are finite.)
Given a group $G$ and a subgroup $H \le G$,
we write $\lbrack G:H\rbrack$ for the index of $H$ in $G$.

Let a group $G$ act on a set $X$. 
As is customary in the study of permutation groups,
given $g\in G$ and $\alpha\in X$, we write $\alpha^g$ for the image of 
$\alpha$ under the action of $g$.  We speak of the {\em orbit} 
$\alpha^A = \{\alpha^g : g\in A\}$ of a point $\alpha$
under the action of a set $A$ of permutations.
Our actions are right actions by default:
$(\alpha^g)^h = \alpha^{gh}$. In consequence, we also use right cosets by
default, i.e., cosets $H g$ (and so $G/H$ is the set of all such cosets).
Clearly $|G/H|=\lbrack G:H\rbrack$.

We define the commutator $\lbrack g,h\rbrack$ by $\lbrack g,h\rbrack = 
g^{-1} h^{-1} g h$. Again, this choice is customary for permutation groups.

Define
\begin{equation}\label{eq:tronk}\begin{aligned}
A_\Sigma &= \{g\in A: \Sigma^g = \Sigma \}, \text{\;\;\;\;\;\;\;\;\;\;\;\;\;\;\;\;\;\;
(the {\em setwise stabilizer})}\\
A_{(\Sigma)} &= \{g\in A: \forall \alpha\in \Sigma \left(\alpha^g = \alpha\right) \}.
\text{\;\;\;\;\; (the {\em pointwise stabilizer})}
\end{aligned}\end{equation}
If $\Sigma=\{g_1,\dotsc,g_m\}$, the setwise stabilizer is denoted by
$A_{\{g_1,\dotsc,g_m\}}$ and the pointwise stabilizer by $A_{(g_1,\dotsc,g_m)}$.

Given a permutation $g\in \Sym(\Omega)$, we define its {\em support} 
$\supp(g)$ to be the set of elements of $\Omega$ moved by $g$:
$\supp(g)=\{\alpha\in \Omega : \alpha^g\neq\alpha\}$.
If a subset $\Delta \subseteq \Omega$ is invariant under $g$, i.e., $\Delta$ is a union of cycles of $g$,
then we define $g|_{\Delta} \in \Sym(\Delta)$ as the {\em restriction} (natural projection) of $g$ to $\Delta$: the permutation $g|_{\Delta}$ acts on $\Delta$ as $g$ does. 
If $\Delta$ is invariant under some $D \subseteq \Sym(\Omega)$ then $D|_{\Delta}=\{ g|_{\Delta} : g \in D \}$.

A partition ${\mathcal{B}}=\{ \Omega_1,\Omega_2,\ldots,\Omega_k \}$ of a
set $\Omega$ ($\Omega_i$ non-empty) 
is called a {\em system of imprimitivity} for a transitive
group $G \le \Sym(\Omega)$ if $G$ permutes the sets $\Omega_i$ for $1 \le
i \le k$. For $|\Omega|\geq 2$, 
a transitive group $G \le \Sym(\Omega)$ is called {\em primitive} if there are only the two trivial systems of imprimitivity for $G$: the partition into one-element sets, and the partition consisting of one part $\Omega_1=\Omega$. 

We say that a graph (or a multigraph) is {\em regular} with {\em degree} or
{\em valency} $d$ if there are $d$ edges adjoining every vertex; 
that is, ``degree''   
and ``valency'' of a vertex mean the same thing. In a directed graph, the {\em out-degree} of a vertex $x$ is the number of edges starting at $x$ while the {\em in-degree} is the number of edges terminating at $x$. 
A directed graph is called {\em strongly connected} if for any two vertices $x,y$, there is a directed path from $x$ to $y$. 

By $f(n)\ll g(n)$, $g(n)\gg f(n)$ and $f(n) = O(g(n))$ we mean one and the
same thing, namely, that there are $N>0$, $C>0$ such that $|f(n)|\leq C\cdot 
g(n)$ for all $n\geq N$.

We write $\log_2 x$ to mean the logarithm base $2$ of $x$ (and {\em not} to mean
$\log \log x$).
\section{Preliminaries on sets, groups and growth}\label{sec:preliminaries}
\subsection{Orbits and stabilizers}\label{subs:orbsta}
The orbit-stabilizer theorem from elementary group theory carries over to
sets. This is a fact whose importance to the area is difficult to
overemphasise. It underlies already \cite{Hel08} at a 
key point (Prop.\ 4.1); the action at stake there is that of a group
$G$ on itself by conjugation.

The setting for the theorem is the action of a group $G$ on a set $X$.
The {\em stabilizer} $G_x$ of a point $x\in X$ is the set
$\{g\in G: x^g = x\}$. 
\begin{lem}[Orbit-stabilizer theorem for sets]\label{lem:orbsta}
Let $G$ be a group acting on a set $X$. Let $x\in X$, and let $A\subseteq G$ be non-empty. 
Then
\begin{equation}\label{eq:applepie}
|A A^{-1}\cap G_x|\geq \frac{|A|}{|x^A|}.\end{equation}
Moreover, for every $B\subseteq G$,
\begin{equation}\label{eq:easypie}
|A B| \geq |A\cap G_x| |x^{B}| .
\end{equation}
\end{lem}
The usual orbit-stabilizer theorem is the special case $A = B = H$, $H$
a subgroup of $G$.
\begin{proof}
 By the pigeonhole principle,
there exists an image $x' \in x^A$ such that 
the set $S = \{a\in A: x^a = x'\}$ has at least 
$|A|/|x^A|$ elements. For any $a, a'\in S$, 
$x^{a (a')^{-1}} = (x')^{(a')^{-1}} = x$. Hence
\[|A A^{-1}\cap G_x| \geq |S S^{-1}|\geq |S|\geq \frac{|A|}{|x^A|}.
\]

Let $b_1, b_2,\dotsc, b_{\ell}\in B$,
$\ell = |x^{B}|$, be elements with $x^{b_i} \ne x^{b_j}$ for $i\ne j$.
Consider all products of the form $a b_i$, $a\in A\cap G_x$,
$1\leq i\leq \ell$. If two such products $a b_i$, $a' b_{i'}$ are equal, then
$x^{b_i} = x^{a b_i} = x^{a' b_{i'}} = x^{b_i'}$. This implies $b_i = b_{i'}$. 
Since $a b_i = a' b_{i'}$, we 
conclude that $a = a'$. We have thus shown that all
 products $a b_i$, $a\in A\cap G_x$, $1\leq i\leq \ell$,
are in fact distinct. Hence
\[\begin{aligned}|A B|&\geq |(A\cap G_x) \cdot \{b_i: 1\leq i\leq \ell\}| \\ &= 
|A\cap G_x| \cdot \ell = |A\cap G_x| \cdot |x^B|.\end{aligned}\]
\end{proof}

As the following corollaries show, 
the relation between the size of $A$, on the one hand, and the size of
orbits and stabilizers, on the other, implies that growth in the size of
either orbits or stabilizers induces growth in the size of $A$ itself.

\begin{cor}\label{cor:grostab}
Let $G$ be a group acting on a set $X$. Let $x\in X$. Let $A\subseteq G$ be a 
non-empty set with $A = A^{-1}$. 
Then, for any $k>0$,
\begin{equation}\label{eq:pear1}
|A^{k+1}| \geq \frac{|A^{k}\cap G_x|}{|A^2\cap G_x|} |A| .
\end{equation}
\end{cor}
\begin{proof}
By (\ref{eq:easypie}),
\[|A^{k+1}| \geq |A^{k} \cap G_x| |x^A| \geq
\frac{|A^{k} \cap G_x|}{|A^2 \cap G_x|}
|A^2 \cap G_x| |x^A| .\]
Since $|A^2 \cap G_x| |x^A| \geq |A|$ (by (\ref{eq:applepie})), we obtain
(\ref{eq:pear1}).
\end{proof}

\begin{cor}\label{cor:groorb}
Let $G$ be a group acting on a set $X$. Let $x\in X$. Let $A\subseteq G$ be a 
non-empty set with $A = A^{-1}$. 
Then, for any $k>0$,
\begin{equation}\label{eq:pear2}
|A^{k+2}| \geq \frac{|x^{A^k}|}{|x^A|} |A| .
\end{equation}
\end{cor}
\begin{proof}
By (\ref{eq:easypie}) and (\ref{eq:applepie}),
\[|A^{k+2}|\geq |A^2\cap G_x| |x^{A^k}| \geq
\frac{|A|}{|x^A|} |x^{A^k}| = \frac{|x^{A^k}|}{|x^A|} |A|.\]
\end{proof}

\subsection{Lemmas on subgroups and quotients}\label{subs:quotients}
We start by recapitulating some of the simple material in \cite[\S 7.1]{HeSL3}.
The first lemma guarantees that we can always find many elements of $A A^{-1}$
in any subgroup of small enough index.
\begin{lem}[{\cite[Lem.\ 7.2]{HeSL3}}]\label{lem:duffy} 
Let $G$ be a group and $H$ a subgroup thereof. Let $A\subseteq G$ be a 
non-empty set. Then
\begin{equation}\label{eq:vento}
|A A^{-1} \cap H| \geq \frac{|A|}{r},\end{equation}
where $r$ is the number of cosets of $H$ intersecting $A$. In particular,
\[|A A^{-1} \cap H| \geq \frac{|A|}{\lbrack G:H\rbrack}.\]
\end{lem}
\begin{proof}
By the orbit-stabilizer principle (\ref{eq:applepie}) applied to the
natural action of $G$ on $G/H$ by multiplication on the right.\footnote{Recall
that we are following the convention that $G/H$ is the set of right cosets
$H g$.} (Set $x=H e=H$.)
\end{proof}

The following two lemmas should be read as follows: growth in a subgroup 
gives growth in the group; growth in a quotient gives growth in the group.

\begin{lem}[{\rm essentially {\cite[Lem.\ 7.3]{HeSL3}}}]\label{lem:koph} 
Let $G$ be a group and $H$ a subgroup thereof. Let $A\subseteq G$ be 
a non-empty set with $A = A^{-1}$.
 Then, for any $k>0$,
\begin{equation}\label{eq:avoc1}
|A^{k+1}| \geq \frac{|A^{k}\cap H|}{|A^2\cap H|} |A| .
\end{equation}
\end{lem}
\begin{proof}
By Cor.\ \ref{cor:grostab} applied to the action of $G$ on
$G/H$ by multiplication on the right (with $x = H e = H$).
\end{proof}

For a group $G$ and a subgroup $H \le G$, we define the coset map $\pi_{G/H}: G \to G/H$ that maps
each $g \in G$ to the right coset $Hg$ containing $g$. 

\begin{lem}[{\rm essentially \cite[Lem.\ 7.4]{HeSL3}}]\label{lem:quotgro} 

Let $A\subseteq G$ be a non-empty set with $A = A^{-1}$. Then, for any $k>0$,
\[|A^{k+2}| \geq \frac{|\pi_{G/H}(A^k)|}{|\pi_{G/H}(A)|} |A| .\]
\end{lem} 
\begin{proof}
By Cor.\ \ref{cor:groorb}, applied with $G$ acting on
$X:=G/H$ by multiplication on the right and with $x:=H$ seen as an element of
$G/H$.
\end{proof}

The following lemma is a generalisation of Lemma~\ref{lem:duffy}.

\begin{lem}
\label{subgroup cosets}
Let $G$ be a group, let $H,K$ be subgroups of $G$ with $H\leq K$, and let
$A\subseteq G$ be a non-empty set. Then
\[|\pi_{K/H}(A A^{-1}\cap K)| \geq \frac{|\pi_{G/H}(A)|}{|\pi_{G/K}(A)|}
\geq \frac{|\pi_{G/H}(A)|}{\lbrack G:K\rbrack}.\]
\end{lem}
In other words: if $A$ intersects $r \lbrack G:H\rbrack$ 
cosets of $H$ in $G$, then $A A^{-1}$ intersects at least 
$r \lbrack G:H\rbrack/\lbrack G:K\rbrack = r \lbrack K:H\rbrack$ 
cosets of $H$ in $K$. (As usual, all our cosets are right cosets.)
\begin{proof}
Since $A$ intersects 
$|\pi_{G/H}(A)|$ cosets of $H$ in $G$ and
$|\pi_{G/K}(A)|$ cosets of $K$ in $G$, and every coset of $K$ in $G$
is a disjoint union of cosets of $H$ in $G$, the pigeonhole principle implies
that there exists a coset $K g$ of $K$ such that $A$
intersects at least $k = |\pi_{G/H}(A)|/|\pi_{G/K}(A)|$ cosets $H a \subseteq K g$. 
Let $a_1,\ldots,a_k$ be elements of $A$ in distinct 
cosets of $H$ in $Kg$. Then $a_i a_1^{-1}\in AA^{-1}\cap K$ for each
 $i=1,\ldots, k$. 
Finally, note that $Ha_1a_1^{-1},\ldots,Ha_ka_1^{-1}$ are $k$ distinct 
cosets of $H$.
\end{proof}





The above lemmas fall into two types: either (a) they reduce the problem of proving
growth in $G$ to that of proving growth in a smaller structure (a
subgroup in Lemma \ref{lem:koph}, a quotient in Lemma \ref{lem:quotgro}),
or
(b) they produce many elements in a smaller structure (a group in Lemma
\ref{lem:duffy}, a quotient in Lemma \ref{subgroup cosets}.

Lastly, a result of a somewhat different nature. It is a version of Schreier's 
lemma (rewritten slightly as in \cite[Lem.\ 2.10]{GH2}). 
Usually, if a set $A$ generates a group $G$, that does not mean that, for
$H$ a subgroup of $G$, the intersection $A\cap H$ will generate
$H$. However, Lemma \ref{schreier} tells us, if $A$ projects onto $G/H$,
then $A^3\cap H$ does generate $H$.
We will use Lemma \ref{schreier}
in the proof of Lemma \ref{lem:sonofschreier} (for $G$ a setwise stabilizer
$(\Sym(n))_{\Delta}$ and $H$ the corresponding
 pointwise stabilizer $(\Sym(n))_{(\Delta)}$). 

\begin{lem}[Schreier]\label{schreier}
Let $G$ be a group and $H$ a subgroup thereof.
 Let $A\subseteq G$ with $A = A^{-1}$ and $e\in A$. Suppose $A$ intersects each
coset of $H$ in $G$. Then $A^3\cap H$ generates $\langle A\rangle \cap H$. Moreover, $\langle A\rangle = \langle A^3\cap H\rangle A$.
\end{lem}
\begin{proof}
Let $C\subseteq A$ 
be a full set of right coset representatives of $H$, with $e\in C$.
We wish to show that $\langle A\rangle = \langle A^3\cap H\rangle C$. 
(This immediately implies both $\langle A\rangle = \langle A^3\cap H\rangle A$
and $\langle A\rangle \cap H = \langle A^3\cap H\rangle$.)

Clearly $e\in \langle A^3\cap H\rangle C$. It is thus 
enough to show that,
if $g = h c$, where $h\in \langle A^3\cap H\rangle$ and $c\in C$,
and $a'\in A$, then $g a'$ still lies in $\langle A^3\cap H\rangle C$. 
This is  easily seen: since $C$ is a full set of coset representatives, there
is a $c' \in C$ with $c' = h' c a'$ for some $h'\in H$, and thus
\[g a' = h c a' = h ((h')^{-1}) h' c a' = h ((h')^{-1}) c' \in
\langle A^3\cap H\rangle (A^3\cap H) C =  \langle A^3\cap H\rangle C,\]
where we use the fact that $h' = c' (a')^{-1} c^{-1} \in A^3$.  
\end{proof}

\subsection{Actions and generators} 
\label{subs:actions}
The proofs of the next two lemmas share a rather simple 
idea. Indeed, both lemmas can be seen as consequences of the well-known 
fact that every connected graph has a spanning tree.\footnote{We thank an
anonymous referee for this comment.} 
 The graph would be the union of the
permutation graphs (with $X$ as the vertex set) induced by the elements of
the set $A$. 

We give two brief proofs without graphs.
\begin{lem}\label{lem:apeman1}
Let $G$ be a group acting transitively on a finite set $X$. Let $A\subseteq G$ with 
$A = A^{-1}$, $e\in A$ and $G=\langle A\rangle$. Then, for any $x\in X$,
\[x^{A^\ell} = X,\]
where $\ell = |X|-1$.
\end{lem}
\begin{proof}
Consider the orbits $\{x\}\subseteq x^A\subseteq x^{A^2}\subseteq
\cdots$. Let $\ell'$ be the smallest integer with $x^{A^{\ell'+1}}=x^{A^{\ell'}}$. As $x^{A^{\ell'+2}} = (x^{A^{\ell'+1}})^A=(x^{A^{\ell'}})^A = x^{A^{\ell'+1}}=x^{A^{\ell'}}$, we have 
$x^{A^{\ell'}}= x^{\langle A\rangle}=x^G=X$. 
Since
\begin{equation*}
\{x\} \subsetneq x^A \subsetneq x^{A^2} \subsetneq \dotsb \subsetneq x^{A^{\ell'}} = X,
\end{equation*}
we have $\ell'\leq |X|-1$.
\end{proof}

\begin{lem}\label{lem:apeman2}
Let $G$ be a group acting transitively on a finite set $X$. Let $A\subseteq G$ with $A = A^{-1}$ and $G=\langle A\rangle$. Then there is a subset $A'\subseteq A$,
$|A'|<|X|$, such that $\langle A'\rangle$ acts transitively on $X$.
\end{lem}
\begin{proof}
Let $x\in X$. Let $A_1 = \{g\}$, where $g$ is any element of $A$ such
that $x^g\ne x$.  For each $i\geq 1$, let $A_{i+1}$ be $A_i\cup \{g_i\}$,
where $g_i$ is an element of $A$ such that $x^{\langle A_i \cup \{g_i\}\rangle}
\supsetneq x^{\langle A_i\rangle}$. If no such element $g_i$ exists, we can 
conclude that $x^{\langle A_i\rangle}$ is taken to 
itself by every $g_i\in A$. This implies that 
$x^{\langle A_i\rangle}$ is taken to  
itself by every product of elements of $A$, and thus
$(x^{\langle A_i\rangle})^{\langle A\rangle} = x^{\langle A\rangle}$ equals 
$x^{\langle A_i\rangle}$. 

Hence, we have a chain
\[\{x\}\subsetneq x^{\langle A_1\rangle} \subsetneq x^{\langle A_2\rangle} 
\subsetneq \cdots \subsetneq x^{\langle A_i\rangle} = x^{\langle A\rangle} = X.\]
Clearly $i\leq |X|-1$, and so $|A_i| \leq |X|-1$. Let $A' = A_i$. 
\end{proof}
\subsection{Large subsets of $\Sym(n)$.} 
\label{subs:large subsets}

Let us first prove a result on large subgroups of $\Sym(n)$.

\begin{lem}
\label{largealt}
Let $n\ge 84$. Let $G \le \Sym(n)$ be transitive, with 
a section isomorphic to $\Alt(k)$ for some $k>n/2$. Then $G$ is either $\Alt(n)$ or $\Sym(n)$.
\end{lem}

\begin{proof}
Since $k\geq 5$, the group $\Alt(k)$ is simple. Hence some composition
factor of $G$ has a section isomorphic to $\Alt(k)$. Assume that $G$ is
imprimitive and let $\mathcal{B}$ be a non-trivial system of imprimitivity
for $G$. Write $b=|\mathcal{B}|$ and $m=n/b$ and let $K$ be the kernel of
the action of $G$ on $\mathcal{B}$. Since $G/K$ is isomorphic to a subgroup 
of $\Sym(b)$, $K$ is isomorphic to a subgroup of $\Sym(m)^b$ and $b,m<k$,
we obtain that $G$ has no section isomorphic to $\Alt(k)$, a contradiction.
This shows that $G$ is primitive.

{}From \cite{MR576980}, we obtain that either $G\ge \Alt(n)$ or $|G|\le 4^n$.
 Since $|G|\geq |\Alt(k)|=k!/2\geq \lceil n/2\rceil!/2$, a direct 
computation shows that the latter case arises only for $n< 84$.
\end{proof}

Our aim for the rest of this subsection will be to show that, if
$A\subset \Sym(n)$ is very large, then $A^{n^{O(1)}}$ contains a copy of 
$\Alt(\Delta)$, $|\Delta|> n/2$.
The next lemma generalizes Bochert's theorem \cite{Boch89}, \cite[Thm.\
3.3B]{DM} to subsets. 
 Recall that,
for $g\in \Sym(\Omega)$, 
we define the {\em support} of $g$ by $\supp(g)=\{\alpha\in \Omega :
\alpha^g\neq\alpha\}$.

\begin{lem}
\label{bochert}
Let $n\ge 5$. Let 
$A \subseteq \Sym([n])$ with $A=A^{-1}$, $e\in A$.
If $\langle A \rangle$ is a primitive permutation
group and  $|A|>n!/(\lfloor n/2 \rfloor !)$, then 
$A^{n^4}$ is either $\Alt([n])$ or $\Sym([n])$.
\end{lem}
This is an example of how one can sometimes
modify a proof of a result about subgroups to give a result about sets:
the proof follows the lines of Bochert's essentially algorithmic
proof,
plus some bookkeeping.
\begin{proof}
Given $A \subseteq \Sym([n])$ as in the statement of the lemma, let $k$
be the smallest integer such that there exists 
$\Delta \subseteq [n]$ with $|\Delta|=k$ and $(A^2)_{(\Delta)}=\{e\}$. Let $\Delta$ be one such set. 

Suppose that $k \leq n/2$. Then 
$\Sym([n])_{(\Delta)}$ has $n!/(n-k)!<|A|$ cosets in $\Sym([n])$. Thus,
by the pigeonhole principle, there exist two distinct elements $a$ and $b$ of $A$ in the same coset. Hence $ab^{-1} \in \Sym([n])_{(\Delta)}$, that is,
$ab^{-1} \in (A^2)_{(\Delta)}$.
This contradicts the definition of $k$. We conclude that $k>n/2$.

The set $\Omega=[n]\setminus \Delta$ has cardinality less than $k$, 
so by definition there exists $g \in (A^2)_{(\Omega)}$ with $g \ne e$. 
Let $\delta\in \Delta$ with $\delta^g\neq\delta$. 
As the set $\Delta \setminus \{ \delta\}$  has cardinality less then 
$k$, by the definition of $k$, there exists $h \in  
(A^2)_{(\Delta \setminus \{ \delta\})}$ with $h\neq e$. Then $\supp(h) \subset \Omega \cup
\{ \delta \}$. Necessarily, $\delta \in \supp(h)$, otherwise 
$(A^2)_{(\Delta)}$ contains the non-identity element $h$. Hence $\supp(g) \cap \supp(h)= \{ \delta \}$ and so the commutator $x = \lbrack g,h \rbrack$ is a $3$-cycle. Note that 
$\lbrack g,h\rbrack \in A^8$. 

Now, since $\langle A \rangle$ is primitive and contains a $3$-cycle, by
Jordan's theorem \cite[Thm.\ 3.3A]{DM} we obtain that 
$\langle A \rangle \ge \Alt([n])$. In particular,
$\langle A \rangle$ is $3$-transitive, and thus its action by conjugation
on the set $X$ of all $3$-cycles is transitive. By Lemma \ref{lem:apeman1},
\[x^{A^{\ell}} = X,\] where $\ell = |X| = n (n-1) (n-2)/3$ and $A^{\ell}$
acts on $x$ by conjugation. Thus
\[A^{n (n-1) (n-2)/3} \lbrack g,h\rbrack A^{n (n-1) (n-2)/3}\]
contains all $3$-cycles in $\Alt([n])$.

Since any element of 
$\Alt([n])$ can be written as a product of at most $\lfloor n/2\rfloor$ $3$-cycles, 
we obtain that $A^{n^4-1}$ contains $\Alt([n])$. Also, if $A$ contains an odd 
permutation, then $A^{n^4}=\Sym([n])$.
\end{proof}

What happens, however, if $\langle A\rangle$ is not transitive, let alone
primitive? We shall see first that, if $A$ is large, then $\langle A\rangle$
must have at least a large orbit. In the following two lemmas,
we use the inequalities
\begin{equation}\label{stirling}
\left(\frac{n}{e}\right)^n < n! < 3 \sqrt{n} \left(\frac{n}{e}\right)^n
\end{equation}

\begin{lem}
\label{factorials}
Let $H<\Sym(n)$ with $|H| \ge d^n n!$, for some number $d$ with $0.5<d<1$. 
If $n$ is greater than a bound depending only on $d$, then 
$H$ has an orbit of length at least $dn$.
\end{lem}
\begin{proof}
Let $k:=\lfloor dn \rfloor$. 
Suppose that the longest orbit length of $H$ is less than $dn$. Then, as is
well-known, $|H|\leq k! (n-k)!$. (The size of a direct product of symmetric
groups $\Sym(\Omega_i)$ only goes up if we pass elements from the smaller
sets $\Omega_i$, $i\geq 2$, to the largest set $\Omega_1$.)



Now, by \eqref{stirling}, we have the following inequalities:
\begin{eqnarray}
\label{factest}
\left(\frac{k}{n}\right)^n \left(\frac{n}{e}\right)^n &<&\left(\frac{k}{n}\right)^n n!\leq d^n n! \le |A|\leq |\langle A\rangle|\leq k!(n-k)! \notag \\
&<& 9\sqrt{k(n-k)} \left(\frac{k}{e}\right)^k 
\left(\frac{n-k}{e}\right)^{n-k}\leq \frac{9}{2} n\frac{k^k(n-k)^{n-k}}{e^n}.
\end{eqnarray}
Simplifying the left-hand side together with the right-hand side, we obtain $k^{n-k}<\frac{9}{2}n(n-k)^{n-k}$, that is, 
$\left( \frac{k}{n-k}\right)^{n-k}<\frac{9}{2} n.$ 

We define $c:=\left( \frac{d}{1-d}\right)^{1-d}$. As 
$$\lim_{n \to \infty}\left( \frac{k}{n-k}\right)^{\frac{n-k}{n}} =c>1,$$
for large enough $n$, depending only on $d$, we have
$\left(\frac{k}{n-k}\right)^{n-k} > \left( \frac{1+c}{2} \right)^n$.
However, $\left( \frac{1+c}{2} \right)^n < \frac{9}{2}n$ is false if $n$ is greater than a bound depending only on $d$, proving our claim.
\end{proof}

Using Bochert's theorem \cite{Boch89}, Liebeck 
 derived a result (\cite[Lem. 1.1]{MR703984}; see \cite[pp. 68--75]{Jor} for
a classical result of the same kind) on large 
subgroups of $\Sym(n)$. It does not assume transitivity or primitivity.
We will generalize it to sets (Prop. \ref{prop:liebest}).
In a somewhat strengthened version \cite[Thm. 5.2B]{DM}, the result from
\cite{MR703984} states the following, among other things: if $H$ is a
subgroup of $\Sym(n)$, $n\geq 9$, and
\begin{equation}\label{eq:esther}
\lbrack \Sym(n) : H\rbrack <
\min\left(\frac{1}{2} \binom{n}{\lbrack n/2\rbrack},
\binom{n}{m}\right)\end{equation}
for some $m\geq n/2$, then there is a set 
$\Delta\subset \lbrack n\rbrack$, $|\Delta| > m$, such that 
\begin{equation}\label{eq:helios}
\Alt(n)_{(\lbrack n\rbrack\setminus \Delta)} \leq H \leq 
\Sym(n)_{\lbrack n\rbrack\setminus \Delta}.\end{equation}

Here, of course, $\Alt(n)_{(\lbrack n\rbrack\setminus \Delta)} \sim
\Alt(\Delta)$ and $\Sym(n)_{\lbrack n\rbrack\setminus \Delta} =
\Sym(n)_\Delta$; in particular, (\ref{eq:helios}) implies that $\Delta$
is an orbit of $\lbrack n\rbrack$.
It is easy to see that, if $|H|\geq d^n n!$, $0.5<d<1$, then 
(\ref{eq:esther}) is fulfilled for $m=\lceil d n\rceil$, provided
that $n$ is larger than a constant depending only on $d$: by Stirling's formula,
\begin{equation}\label{eq:sinam}
\binom{n}{\lceil d n\rceil} \gg \frac{1}{\sqrt{n}}
\frac{n^n}{\lceil dn\rceil^{\lceil dn\rceil} \lfloor (1-d) n\rfloor^{\lfloor
    (1-d)n\rfloor}} \gg \frac{1}{n^{3/2}} \left(
\frac{1}{d^d (1-d)^{1-d}}\right)^n,\end{equation}
and, since $d^d (1-d)^{1-d} < d$ for $d\in (1/2,1)$, this is certainly
greater than $(1/d)^n$ for $n$ large enough. The inequality
$\frac{1}{2} \binom{n}{\lbrack n/2\rbrack} \gg 2^n/\sqrt{n}$ implies
$\frac{1}{2} \binom{n}{\lbrack n/2\rbrack} >
(1/d)^n$ immediately for all large $n$. Thus (\ref{eq:helios}) holds
for some $\Delta$ with $|\Delta| > dn$.

We will show an analogue of (\ref{eq:helios}) holds for a set $A$ instead
of a subgroup $H$ (Prop. \ref{prop:liebest}). This can be shown in two
ways: we can use Liebeck's result (\ref{eq:helios}) for groups, or we 
can give an elementary proof using only counting arguments. (Both 
\cite{MR703984} and \cite{DM} do a detailed examination of the subgroup
structure of $\Sym(n)$ in order to give a result valid for small $n$.)


Let us first give an elementary proof of a somewhat weaker statement.
\begin{lem}\label{lem:lieb}
Let $d$ be a number with $0.5<d<1$.
If $A \subseteq \Sym([n])$  (with $A=A^{-1}$) has cardinality $|A| \ge d^n n!$ and $n$ is larger than a bound depending only on $d$, 
then there exists an orbit $\Delta \subseteq [n]$ of $\langle A\rangle$ 
such that 
$|\Delta|\ge dn$ and $(A^{n^4})|_{\Delta}$ 
is $\Alt(\Delta)$ or $\Sym(\Delta)$.
\end{lem}
\begin{proof}
By Lemma~\ref{factorials}, for large enough $n$ the group 
$\langle A \rangle$ has an orbit $\Delta$ of
length $k \ge dn$. 
Write $\rho=k/n$ and note that $d\leq \rho\leq 1$. 
The group $G = B|_\Delta$ has order at least $d^n n!/(n-k)!$, 
so estimating $k!(n-k)!$ from above as in \eqref{factest} and estimating $n!$ from below by \eqref{stirling}, we obtain 
\begin{eqnarray}
\label{liebupper}
\lbrack \Sym(\Delta):G \rbrack \le \frac{k!(n-k)!}{d^n n!} <
\frac{9}{4} n\frac{k^k(n-k)^{n-k}}{d^n n^n} = \notag \\
\frac{9}{4} n \left(\frac {\rho^{\rho} (1 -\rho)^{1-\rho}}{d} \right)^n = \frac{9}{4} n \left( 2^{\frac{1}{\rho}}\rho (1-\rho)^{\frac{1-\rho}{\rho}} \right)^{\rho n} \left(\frac{1}{2d} \right)^n.
\end{eqnarray}

Next, we show that for large values of $n$ the transitive group $G$ cannot be imprimitive. Indeed, if $G$ is imprimitive, then using \eqref{stirling} we have 
\begin{equation}
\label{lieblower}
\lbrack \Sym(\Delta):G \rbrack \geq \frac{1}{2}{k\choose \lfloor k/2\rfloor} > 
\frac{1}{2}\frac{(\frac{k}{e})^k}{9k(\frac{k}{2e})^k}>\frac{1}{18n}2^{\rho n}.
\end{equation}
A direct computation shows that the function $f(\rho)=2^{1/\rho}\rho(1-\rho)^{(1-\rho)/\rho}$ is monotone increasing in the interval $[1/2,1)$ with supremum $2$. Hence, comparing the upper and lower bounds for $\lbrack \Sym(\Delta):G \rbrack$ deduced in \eqref{liebupper} and \eqref{lieblower}, we obtain \begin{equation}
\label{liebcomb}
\frac{9}{4} n 2^{\rho n} \left(\frac{1}{2d} \right)^n > \frac{1}{18n}2^{\rho n}.
\end{equation}
As  $d>1/2$, for large enough $n$ we have $(2d)^n >(18n)(\frac{9}{4}n)$ and therefore \eqref{liebcomb} cannot hold.

Hence $G$ is primitive and $A|_{\Delta}$ is a set
of size at least $d^n n!/(n-k)! \ge d^n k! > k!/(\lfloor k/2\rfloor)!$ (where the last inequality holds for $n$ greater than a lower bound depending only on $d$). Therefore, by Lemma \ref{bochert}, $(A|_\Delta)^{n^4}$ is either $\Alt(\Delta)$
or $\Sym(\Delta)$, and hence so is $(A^{n^4})|_\Delta = (A|_\Delta)^{n^4}$.
\end{proof}

Now we get the full analogue of (\ref{eq:helios}).
\begin{prop}\label{prop:liebest}
Let $d$ be a number with $0.5<d<1$. Let
$A \subseteq \Sym(n)$ with  $A=A^{-1}$ and $e\in A$. If $|A| \ge d^n n!$ and 
$n$ is larger than a bound depending only on $d$, 
then there exists an orbit $\Delta \subseteq [n]$ of $\langle A\rangle$ 
such that
$|\Delta|\ge dn$ and $(A^{8 n^5})_{(\lbrack n\rbrack\setminus \Delta)}|_{\Delta}$ 
contains $\Alt(\Delta)$.
\end{prop}
\begin{proof}
By Lemma \ref{lem:lieb}, there is an orbit $\Delta$ of $\langle A\rangle$
such that $|\Delta|\geq dn$ and $(A^{n^4})|_\Delta$ is $\Alt(\Delta)$
or $\Sym(\Delta)$. Let $A' = A^{n^4}$.

It is clear that $|A'|\geq |\Alt(\Delta)|>|\Sym(\lbrack n\rbrack \setminus
\Delta)|$. Thus, by the pigeonhole principle, there are $h_1, h_2\in A'$,
$h_1\ne h_2$, such that $h_1|_{\lbrack n\rbrack \setminus \Delta} = 
h_2|_{\lbrack n\rbrack \setminus \Delta}$, and so
$g = h_1 h_2^{-1}$ fixes $\lbrack n\rbrack\setminus \Delta$ pointwise.

We show that $((A')^{14})_{([n]\setminus\Delta)}$ contains an element $g'$ such that $g'|_\Delta$ is a $3$-cycle. If $g|_{\Delta}$ has at least two fixed points then there exists an element $h \in A'$ such that $h|_{\Delta}$ is a $3$-cycle, with 
$\supp(h|_{\Delta})$ intersecting $\supp(g|_{\Delta})$ in exactly one
point. Then $g'=[g,h] \in (A')^{2+1+2+1}=(A')^6$ fixes 
$\lbrack n\rbrack\setminus \Delta$ pointwise and $g'|_{\Delta}$ is a $3$-cycle.
If $g$ contains a cycle $(\alpha \beta \gamma\delta \dotsc)$ of length at least $4$, then
we choose an element $h\in A'$ with $h|_\Delta = (\alpha \beta \gamma)$ and let
 $g' = \lbrack g,h\rbrack \in (A')^6$. Then $g'$ fixes $\lbrack n\rbrack \setminus
\Delta$ pointwise and $g'|_{\Delta}$ is the $3$-cycle $(\alpha \beta \delta)$.

In all other cases, $|\supp(g|_{\Delta})| \ge |\Delta|-1 \ge 6$ 
(assuming $n\geq 13$, which implies $|\Delta|\geq 7$) and 
all nontrivial cycles of $g$ have length $2$ or $3$. Hence $g|_{\Delta}$ contains
at least two $3$-cycles, or at least two $2$-cycles. 

If $g|_{\Delta}$ contains
the cycles $(\alpha\beta\gamma)$ and $(\delta\eta\nu)$ then we choose an element $h\in A'$ with $h|_{\Delta}=(\alpha\eta)(\beta\delta\gamma\nu)$. A little computation shows that $g'=\lbrack g,h\rbrack$ fixes $[n]\setminus \Delta$ pointwise and $g'|_{\Delta}$ is the $3$-cycle $(\delta\eta\nu)$. 

Finally, suppose $g$ contains the $2$-cycles $(\alpha \beta)$ and $ (\gamma \delta)$. 
We choose again an element 
$h\in A'$ with $h|_{\Delta} = (\alpha \beta \gamma)$; then 
$\supp(\lbrack g,h\rbrack) = \{\alpha,\beta,\gamma,\delta\}$ and 
$\lbrack g,h\rbrack$ fixes $\lbrack n\rbrack\setminus \Delta$ pointwise.
Since $\lbrack g,h\rbrack \in (A')^6$ 
also fixes at least two points of $\Delta$,
we deduce as in the very first case of our analysis that the commutator 
$g' =\lbrack \lbrack g,h \rbrack, h' \rbrack$ with an appropriate $h' \in
A'$ 
is a $3$-cycle. Note that $g' \in (A')^{6+1+6+1}=(A')^{14}$.  

Given any $3$-cycle $s$
in $\Sym(\Delta)$, we can conjugate $g'$
by an appropriate element of $A'$ 
to get an element of $((A')^{16})_{(\lbrack n\rbrack \setminus \Delta)}$ whose restriction to $\Delta$ equals $s$.
Now, every element of $\Alt(\Delta)$ is the
product of at most $\lfloor |\Delta|/2\rfloor$ $3$-cycles. Hence
$((A')^{16 \lfloor n/2\rfloor})_{(\lfloor n\rfloor \setminus \Delta)}|_{\Delta}$ 
contains $\Alt(\Delta)$. 
\end{proof}

An anonymous referee kindly provides the following argument, showing 
that Prop.\ \ref{prop:liebest}, which is a generalization of (\ref{eq:helios}),
can be proven {\em using} (\ref{eq:helios}). 
\begin{proof}[Second proof of Prop.\ \ref{prop:liebest}] (This proof 
gives Prop. \ref{prop:liebest} with $A^{2 (n^4+1) n^4}$ instead of $A^{8 n^5}$.)
By (\ref{eq:helios}) applied to $H= \langle A\rangle$, there is 
a set $\Delta$ with $|\Delta|>dn$ such that (a) $H$ is contained in 
$\Sym(n)_{\Delta}$ and (b) $H$ contains the subgroup 
$D = \Alt(n)_{(\lbrack n\rbrack\setminus
  \Delta)}$, i.e., $H|_\Delta$ contains $\Alt(\Delta)$. Let $B = A^2\cap D$. By
 Lemma \ref{lem:duffy},
\[|B| =|A^2\cap D|\geq \frac{|A|}{\lbrack \Sym(n)_\Delta :D\rbrack} \geq
 \frac{d^n n!}{2 (n-|\Delta|)!} > \frac{d^n |\Delta|!}{2}
 \geq \frac{|\Delta|!}{2^{n+1}} \geq \frac{|\Delta|!}{2^{2|\Delta|+1}}.
\]
For $n$ sufficiently large (and hence $|\Delta|$ sufficiently large),
$2^{2 |\Delta|+1} < \lfloor |\Delta|/2\rfloor!$, and so we obtain that
$|B|> |\Delta|!/\lfloor |\Delta|/2\rfloor!$.

Since $\langle A|_\Delta\rangle = H|_\Delta$ contains $\Alt(\Delta)$, 
$\langle (A\cup B)|_\Delta\rangle$ is $\Alt(\Delta)$ or $\Sym(\Delta)$ 
-- and, in particular, it is primitive. Hence,
a first application of Lemma \ref{bochert} (with $\Delta$ instead
of $\lbrack n\rbrack$)
implies that $((A\cup B)|_\Delta)^{n^4}$ is $\Alt(\Delta)$ or
$\Sym(\Delta)$. 

The set $S = \{g b g^{-1} : g \in (A\cup B)^{n^4}, b\in B\}$ is in $D$; 
moreover, $\langle S|_\Delta\rangle$ is normal in $\Alt(\Delta)$.
Since $S_\Delta$ is non-trivial (by $|B|>1$), we conclude that
$\langle S|_\Delta\rangle =\Alt(\Delta)$. Now we apply Lemma \ref{bochert}
(again with $\Delta$ instead of $\lbrack n\rbrack$) and obtain that
$(S|_\Delta)^{n^4} = \Alt(\Delta)$. Since $S^{n^4}\subset A^{(2n^4+2)
  n^4}$,
we are done. 
\end{proof}

\subsection{Bases and stabilizer chains}
\label{subs:bases}
Given a permutation group $G$ on a set $\Omega$, a subset $\Sigma$ of $\Omega$ is called a {\em base} if $G_{(\Sigma)} = \{e\}$. 
This definition goes back to Sims 
\cite{MR0257203}. If, instead of $G$, we consider a subset $A$ of $\Sym(\Omega)$, then, as the following
lemma suggests, it makes sense to see whether $(A A^{-1})_{(\Sigma)}$ (rather
than $A_{(\Sigma)}$) equals $\{e\}$.

\begin{lem}\label{lem:wort}
Let $A\subseteq \Sym(\Omega)$, $|\Omega|=n$. If $\Sigma\subseteq \Omega$ is such
that $(A A^{-1})_{(\Sigma)} = \{e\}$, then $|\Sigma|\geq \log_n |A|$.
\end{lem}
\begin{proof}
Notice first that $\lbrack \Sym(\Omega): (\Sym(\Omega))_{(\Sigma)}\rbrack
\leq n^{|\Sigma|}$. By the pigeonhole principle, if $|A|>n^{|\Sigma|}$, then there exists a right coset of $(\Sym(\Omega))_{(\Sigma)}$ containing more than one element
of $A$, and thus
\[|(A A^{-1})_{(\Sigma)}| = |A A^{-1} \cap (\Sym(\Omega))_{(\Sigma)}| > 1.\]
Hence, if $(A A^{-1})_{(\Sigma)} = \{e\}$, then we have $|A|\leq n^{|\Sigma|}$, i.e.,
$|\Sigma|\geq \log_n |A|$.
\end{proof}

The use of stabilizer chains $H > H_{\alpha_1} > H_{(\alpha_1,\alpha_2)} > \dotsb$ is
very common in computational group theory (starting, again, with the work of
Sims; see references in \cite[\S 4.1]{MR1970241}).  We may study a similar chain $A > A_{\alpha_1} > A_{(\alpha_1,\alpha_2)} > \dotsb$ when $A$ is merely a set. 

\begin{lem}\label{lem:hog}
Let $\Sigma =\{ \alpha_1, \ldots, \alpha_m \} \subseteq [n]$ 
and $A \subseteq  \Sym([n])$. Suppose that 
\[
\left| \alpha_i^{A_{(\alpha_1,\ldots,\alpha_{i-1})}} \right| \geq r_i
\]
for all $i=1,2,\dotsc, m$. 
Then $A^m$ intersects at least $\prod_{i=1}^{m} r_i$ cosets of 
$\Sym([n])_{(\Sigma)}$.
\end{lem}
\begin{proof}
For each $1 \le i \le m$, write
$\Delta_i=\alpha_i^{A_{(\alpha_1,\ldots,\alpha_{i-1})}}$; thus $|\Delta_i|\geq r_i$. For each 
$\delta\in\Delta_i$, pick 
$g_\delta\in A_{(\alpha_1,\ldots,\alpha_{i-1})}$ with 
$\alpha_i^{g_\delta}=\delta$ and write 
$S_i=\{g_\delta: \delta\in \Delta_i\}$. 
Clearly, $|S_i|=|\Delta_i|$ and $S_i\subseteq A$. We show that for every two 
distinct tuples  
\[
(s_1,s_2,\ldots,s_m),(s'_1,s'_2,\ldots,s'_m) \in S_1 \times \cdots \times S_m
\]
 the products
$P=s_ms_{m-1}\cdots s_1$ and $P'=s'_ms'_{m-1}\cdots s'_1$ belong to two distinct cosets of 
$\Sym([n])_{(\Sigma)}$. From this it follows that $A^m$ intersects at least $|S_1|\cdots |S_m| = |\Delta_1|\cdots |\Delta_m| \geq \prod_{i=1}^m r_i$
cosets of $\Sym([n])_{(\Sigma)}$.

We argue by contradiction, that is, we assume that $P$ and $P'$ map $(\alpha_1,\ldots,\alpha_m)$ to the same $m$-tuple. Let $j$ be the smallest index such that $s_j \ne s'_j$. Then $Q=Ps_1^{-1}\cdots s_{j-1}^{-1}$ and
$Q'=P' s_1^{-1}\cdots s_{j-1}^{-1}=P' {s'}_1^{-1}\cdots {s'}_{j-1}^{-1}$ also map
$(\alpha_1, \ldots, \alpha_m)$ to the same $m$-tuple. Note that for all 
$k \le m$, $s_k$ and $s'_k$ fix $(\alpha_1, \ldots, \alpha_{k-1})$ pointwise.
Thus 
\[
\alpha_j^Q=\alpha_j^{s_j} \ne \alpha_j^{s'_j}=\alpha_j^{Q'},
\]
contradicting our assumption.
\end{proof}

We thus see that, if we choose $\alpha_1,\alpha_2,\dotsc$ so that the
orbits $\alpha_i^{A_{(\alpha_1,\ldots,\alpha_{i-1})}}$ are large, we get to occupy
many cosets of $(\Sym([n]))_{(\Sigma)}$. By Lemma \ref{subgroup cosets}, 
this will enable
us to occupy many cosets of $(\Sym([n]))_{(\Sigma)}$ in the setwise stabilizer
$(\Sym([n]))_{\Sigma}$. We will then be able to apply Prop.\ \ref{prop:liebest}
to build a large alternating group within 
$\Sym(\Sigma) \cong (\Sym([n]))_{\Sigma}/(\Sym([n]))_{(\Sigma)}$. This procedure
is already implicit in \cite[Lem.\ 3]{Pyb93}; indeed, what amounts to
this is signalled by Pyber as the main new element in his refinement 
\cite[Thm.\ A]{Pyb93} of Babai's theorem on the order of doubly transitive
groups \cite{Bab82}. The main difference is that we have to work, of course,
with sets rather than groups; we also obtain a somewhat stronger
conclusion due to our using Prop.\ \ref{prop:liebest} rather than invoking
Liebeck's lemma directly.

\begin{lem}\label{lem:brioche}
Let $A\subseteq \Sym([n])$ with $A = A^{-1}$ and $e\in A$. 
Let $\Sigma =\{ \alpha_1, \ldots, \alpha_m \} \subseteq [n]$ be such that
\begin{equation}\label{eq:largorb}
\left| \alpha_i^{A_{(\alpha_1,\ldots,\alpha_{i-1})}} \right| \ge d n
\end{equation}
for all $i=1,2,\dotsc, m$, where $d>0.5$. Then, provided that $m$ is
larger than a bound $C(d)$
 depending only on $d$, there exists $\Delta \subseteq \Sigma$ with 
$|\Delta| \ge d|\Sigma|$ and 
\[\Alt(\Delta)\subseteq ((A^{16 m^6})_{\Sigma})_{(\Sigma\setminus \Delta)}|_\Delta.\]
\end{lem}
\begin{proof}
By (\ref{eq:largorb}) and Lemma \ref{lem:hog}, $A^m$ intersects at least
$(d n)^m$ cosets of $\Sym([n])_{(\Sigma)}$ in $\Sym([n])$. Since 
\[\lbrack \Sym([n]) : \Sym([n])_{\Sigma}\rbrack = \frac{\lbrack \Sym([n]) : \Sym([n])_{(\Sigma)}\rbrack}{
\lbrack \Sym([n])_{\Sigma} : \Sym([n])_{(\Sigma)}\rbrack}
\leq \frac{n^m}{m!},\] 
Lemma \ref{subgroup cosets} implies (with $G = \Sym([n])$, $K = 
\Sym([n])_{\Sigma}$, $H = \Sym([n])_{(\Sigma)}$, and $A^m$ instead of $A$) that
\[|\pi_{K/H}(A^{2m}\cap K)| \geq \frac{|\pi_{G/H}(A^m)|}{n^m/m!} \geq
\frac{(d n)^m}{n^m/m!} = d^m m! .\]
Note that $|\pi_{K/H}(A^{2m}\cap K)| = \left|(A^{2m})_{\Sigma}|_\Sigma \right|$.
We can thus apply Prop.\ \ref{prop:liebest} (with $m$ instead of $n$, and
$A' = (A^{2m})_{\Sigma}|_\Sigma$ instead of $A$)
 and obtain that there is a set
$\Delta \subseteq \Sigma$ such that
$|\Delta|\geq d m$ and
$((A')^{8 m^5})_{(\Sigma \setminus \Delta)}|_{\Delta}$ 
contains $\Alt(\Delta)$.
\end{proof}

\subsection{Existence of elements of small support}
\label{subs:small support}
The following lemma is essentially
\cite[Lemma 3]{MR894827} (or \cite[Lemma 1]{BS88}; see also \cite{BLS87}). 

\begin{lem}\label{lem:small support}
Let $\Delta\subseteq \lbrack n\rbrack$, $|\Delta|\geq c (\log n)^2$, $c>0$. 
Let $H \leq (\Sym(n))_\Delta$. Assume $H|_\Delta$ is $\Alt(\Delta)$ or $\Sym(\Delta)$.

 Let
$\Gamma$ be any orbit of $H$. Then, if $n$ is larger than a bound depending
only on $c$, $H$ contains an element $g$ with
$g|_{\Delta} \ne 1$ and $|\supp(g|_{\Gamma})|<|\Gamma|/4$.
\end{lem}
\begin{proof}
Let $p_1=2$, $p_2=3$,\dots, $p_k$ be the sequence of the first $k$ primes, where
$k$ is the least integer such that $p_1 p_2 \cdots p_k>n^4$.
Much as in \cite{MR894827}, we remark that, by elementary bounds
towards the prime number theorem, 
\begin{equation}\label{bound2}
2p_1+p_2+\cdots+p_k < c (\log n)^2,\end{equation}
 provided that $n$ be larger than a bound depending only on $c$. 
Thus $H$ contains an element $h$ such that $h|_{\Delta}$ consists of $|\Delta|-(2p_1+p_2+\cdots + p_k)$ fixed points
and cycles of length $p_1,p_1,p_2,p_3,\ldots,p_k$. (We need two cycles
of length $p_1 = 2$ because we want an even permutation on $\Delta$.)

 We can now
reason as in \cite[Lemma 3]{MR894827} or \cite[Lemma 1]{BS88}.
For every $\gamma\in \Gamma$, denote by $\kappa_\gamma$ 
the length (possibly $1$) 
of the cycle of $h$ containing $\gamma$ and for $i \le k$ define 
$\Gamma_i:= \{ \gamma \in \Gamma : p_i \mid \kappa_{\gamma} \}$. Then 
\begin{equation}
\label{eq:szeged}
\sum_{\gamma \in \Gamma} \sum_{p_i \mid \kappa_\gamma} \log p_i < |\Gamma| \log n
\end{equation}
because $\kappa_\gamma <n$ implies that for all $\gamma$ the inner sum is less than $\log n$. Exchanging the order of summation, 
\[
\sum_{\gamma \in \Gamma} \sum_{p_i \mid \kappa_\gamma} \log p_i = \sum_{i=1}^k |\Gamma_i| \log p_i.
\]
If $|\Gamma_i|\ge |\Gamma|/4$ for all $i \le k$ then 
\[
\sum_{i=1}^k |\Gamma_i| \log p_i \ge \frac{|\Gamma|}{4} \log \left( \prod_{i=1}^k p_i \right) > \frac{|\Gamma|}{4} \log (n^4) = |\Gamma| \log n,
\]
contradicting \eqref{eq:szeged}.
Hence there is
a prime $p\leq p_k$ such that $p|\kappa_{\gamma}$ for fewer than 
$|\Gamma|/4$ elements $\gamma$ of $\Gamma$. Denoting the order of $h$ by 
$|h|$, we define
$g = h^\ell$ for $\ell :=|h|/p$. We obtain that 
$|\supp(g|_{\Gamma})| <|\Gamma|/4$.
We also have that $g$ is non-trivial, since $g|_{\Delta}$ contains
a $p$-cycle.
Clearly $g\in H$, and so we are done.\footnote{Since we need only the
existence of $g$ for the moment, we are not concerned by the fact that
$l$ is very large. Compare this to the situation in \cite{BS88},
where the use of a large $l$ causes diameter bounds much weaker than those
in the present paper.}
\end{proof}

\section{Random walks and generation}\label{sec:random walks}
\subsection{Random walks}\label{subs:random walks}
The aim of this subsection is to present some basic material on random
walks. As stated in the outline, our later use of random walks to mimic the
uniform distribution in combinatorial arguments is clearly
influenced by \cite{BBS04}; indeed, this subsection is very close to 
the first two thirds of \cite[\S 2]{BBS04}.

Let $\Gamma$ be a strongly connected directed multigraph with vertex set $V = 
V(\Gamma)$. For $x\in V(\Gamma)$, we denote by $\Gamma(x)$ the multiset of endpoints of the edges starting at $x$ (counted with multiplicities in case of multiple edges). We are interested in the special case when $\Gamma$ is 
\emph{regular of valency $d$} (i.e., $|\Gamma(x)|=d$, for each $x\in V(\Gamma)$) and $\Gamma$ is also {\em symmetric} in the sense that for all vertices $x,y \in V(\Gamma)$, the number of edges connecting $x$ to $y$ is the same as the number of edges connecting $y$ to $x$. These two conditions imply that the adjacency matrix $A$ of $\Gamma$ is symmetric and all row and column sums are equal to $d$.
 
A {\em lazy random walk} on $\Gamma$ is a stochastic process where a particle moves from vertex to vertex; if the particle is at vertex $x$ such that $\Gamma(x)=\{y_1,\dots, y_d\}$, then the particle 
\begin{itemize}
\item stays at $x$ with probability $\frac12$;
\item moves to vertex $y_i$ with probability $\frac1{2d}$, for all $i=1,\dots, d$.
\end{itemize}

Here we are concerned with the asymptotic rate of convergence for the probability distribution of a particle in a
lazy random walk on $\Gamma$. For $x,y\in V(\Gamma)$, write $p_k(x,y)$ for the 
probability that the particle is at vertex $y$ after $k$ steps of a lazy random
walk starting at $x$. 
For a fixed $\varepsilon>0$, the {\em $\ell_\infty$-mixing time for $\varepsilon$} 
is the minimum value of $k$ such that
$$\frac{1}{|V(\Gamma)|}(1-\varepsilon)\leq p_k(x,y) \leq \frac{1}{|V(\Gamma)|}(1+\varepsilon)$$
for all $x,y\in V(\Gamma)$.

We can give a crude (and well-known; see, e.g., \cite[Fact 2.1]{BBS04}) 
upper bound on the $\ell_\infty$ mixing time for regular symmetric multigraphs
in terms of $N=|V(\Gamma)|$, $\varepsilon$ and the valency $d$ alone.

\begin{lem}\label{l: mixing time}
Let $\Gamma$ be a connected, regular and symmetric multigraph of valency $d$ and with $N$ vertices. Then the $\ell_\infty$ mixing time for $\varepsilon$ is at most $N^2 d \log(N/\varepsilon)$.  
\end{lem}
\begin{proof}
Let $A$ be the adjacency matrix of $\Gamma$. Since $A$ is symmetric, the 
eigenvalues of $A$ are real; moreover, 
their modulus is clearly no more than $d$ in magnitude. 
Let \[d=\mu_1\geq \mu_2\geq \cdots\geq \mu_N\geq -d\] be the eigenvalues of $A$
 and write $P=I/2+A/2d$, where $I$ is the $N\times N$-identity matrix. The 
matrix $P$ is the probability transition matrix for the Markov process 
described by a lazy random walk on $\Gamma$. 

The sum of every row or column of $P$ is $1$, i.e., 
$P$ is a {\em doubly stochastic} matrix. The eigenvalues of $P$ are 
\[1=\lambda_1\geq \lambda_2\geq \cdots \geq\lambda_N\geq 0\] with 
$\lambda_i=1/2+\mu_i/2d$ for each $i=1,\ldots,N$. It is well-known that the 
asymptotic rate 
of convergence to the uniform distribution of a lazy random walk is determined 
by $\lambda_2$: since $P$ is symmetric, there is a basis of $\mathbb{R}^N$
consisting of orthogonal eigenvectors $v_1,v_2,\dotsc,v_n$ of $P$ with
eigenvalues $\lambda_1,\dotsc,\lambda_N$, where every eigenvector $v_i$
has $\ell_2$-norm $1$ with respect to (say)
the counting measure; writing $e_x$ for the probability distribution having
value $1$ at $x$ and $0$ elsewhere, we see that (by Cauchy-Schwarz and
Plancherel)
\[\sum_{j=1}^N \left|\langle e_x,v_j\rangle\right| 
\left|\langle v_j,e_y\rangle\right| \leq 
\sqrt{\sum_{j=1}^N |\langle e_x,v_j\rangle|^2}
\sqrt{\sum_{j=1}^N |\langle v_j,e_y\rangle|^2} \leq
|e_x|_2 \cdot |e_y|_2 = 1,\]
and, since
\[\begin{aligned} p_k(x,y) &= \langle P^k e_x, e_y\rangle = 
\langle \sum_{j=1}^N \langle e_x, v_j\rangle\cdot  P^k v_j, e_y\rangle = 
\sum_{j=1}^N \langle e_x,v_j\rangle \cdot \lambda_j^k \langle v_j,e_y\rangle\\
&= \frac{1}{\sqrt{N}} \cdot 1^k \cdot \frac{1}{\sqrt{N}} + 
\sum_{j=2}^N \langle e_x,v_j\rangle \cdot \lambda_j^k \langle v_j,e_y\rangle,
\end{aligned}\]
we see that
\[\left|p_k(x,y) - \frac{1}{N}\right| \leq \lambda_2^k
 \sum_{j=2}^N \left|\langle e_x,v_j\rangle\right| 
\left|\langle v_j,e_y\rangle\right| \leq \lambda_2^k.\]

By~\cite[Lemma~$2.4$ and Theorem~$3.4$]{MR0573021}, we have 
\[\lambda_2\leq 1-2(1-\cos(\pi/N))\mu(P),\] where 
$\mu(P)=\min_{\emptyset\neq M\subseteq V}\sum_{i\in M,j\notin M}p_{ij}$. As $\Gamma$ 
is a connected regular graph of valency $d$, we have $\mu(P)\geq 1/2d$. Using 
the Taylor series for $\cos(x)$, we see 
that $(1-\cos(\pi/N))\geq 1/N^2$. Thence $|p_k(x,y)-1/N|\leq (1-1/(N^2d))^k$. 
Since $1-x\leq e^{-x}$ for all $x$, we obtain
$|p_k(x,y)-1/N|\leq \varepsilon/N$ for $k\geq N^2d\log(N/\varepsilon)$,
 as desired.
\end{proof}

We will generally study regular symmetric multigraphs of the following
type. (The following argument is already present in \cite[\S 2]{BBS04};
indeed, the only difference between Lemma \ref{l:ktuples} here and corresponding
material in \cite[\S 2]{BBS04} is that Lemma \ref{l:ktuples} applies to
ordered as opposed to unordered $k$-tuples.)
Let $G$ be a group and $A$ be a subset of $G$ with $A = A^{-1}$ and $e\in
A$. 
Let $G$ act on a set $X$. We take the
elements of $X$ as the vertices of our multigraph, and draw one edge from
$x\in X$ to $x' \in X$ for every $a\in A$ such that $x^a = x'$. 
A walk on the graph then corresponds to the action of an element of $A^{\ell}$
on an element $x$ of $X$, where $\ell$ is the length of the walk and $x$ is 
the starting point of the walk. 

Lemma \ref{l: mixing time} 
then gives us a lower bound on how large $\ell$ has to be for the
action of $A^{\ell}$ on $X$ to have a rather strong randomising effect.

\begin{lem}\label{l:ktuples}Let $H$ be a $k$-transitive subgroup of $\Sym([n])$. Let $A$ be a set of generators of $H$ with $A=A^{-1}$ and
$e\in A$. 
Then there is a subset $A'\subseteq A$ with $A' = (A')^{-1}$, such that, for every
$\varepsilon > 0$, for any 
$\ell\geq 2 n^{3k}\log(n^k/\varepsilon)$, and for any $k$-tuples $\overline{x}=(x_1,\ldots,x_k)$, $\overline{y}=(y_1,\ldots,y_k)$ of distinct elements of $[n]$, the probability of the event $$\overline{y}=\overline{x}^{g_1g_2\cdots g_\ell}$$
for $g_1,\ldots,g_\ell\in A'$ (chosen independently, with uniform distribution on $A'\setminus \{e\}$ and with the identity being assigned probability $1/2$) is at least $(1-\varepsilon)\frac{(n-k)!}{n!}$ and at most $(1+\varepsilon)\frac{(n-k)!}{n!}$.
\end{lem}
\begin{proof}
Let $\Delta$ be the set of $k$-tuples of distinct elements of $[n]$. Since 
$H$ acts transitively on $\Delta$ and since $\langle A\rangle = H$,
Lemma \ref{lem:apeman2}
 gives us a subset $A'$ of $A$ with $\langle A'\rangle$ transitive on 
$\Delta$ and with $|A'|< |\Delta|$.
 Set $A_0=A'\cup A'^{-1}$. Let $\Gamma$ be
 the multigraph with vertex set $\Delta$ and with 
$\Gamma(\overline{x})=\{\overline{x}^a\mid a\in A_0\}$
as the multiset of neighbours of $\overline{x}$ for each 
$\overline{x}\in \Delta$. Clearly, $\Gamma$ is a regular graph of valency 
$|A_0|\leq 2|\Delta|$ and with $|\Delta|\leq n^k$ vertices. Now the statement
follows from Lemma~\ref{l: mixing time} applied to $\Gamma$.
\end{proof}

\subsection{Generators} \label{subs:generators}
Given $A\subseteq \Sym(\lbrack n \rbrack)$ such that $\langle A \rangle$
is $\Alt(\lbrack n \rbrack)$ or $\Sym(\lbrack n \rbrack)$, how long can it take to construct a {\em small} set of generators for a {\em transitive} subgroup of $\langle A\rangle$? This subsection is devoted to answering that question. We start by proving two auxiliary lemmas.

\begin{lem}\label{lem:macanus}
Let $A\subset \Sym([n])$, $e\in A$. Assume $\langle A\rangle$ is transitive. Then there is a $g\in A^n$ such that
$|\supp(g)|\geq n/2$.
\end{lem}
\begin{proof}
For each $i\in \lbrack n\rbrack$, let $g_i$ be an element of $A$ moving $i$. (If no such element existed, then  $\langle A\rangle$ could not
be transitive.) Let $g = g_1^{r_1} g_2^{r_2} \dotsc g_n^{r_n}$, where $r_1, r_2,\dotsc, r_n\in \{0,1\}$ are independent random variables taking
the values $0$ and $1$ with equal probability.\footnote{Such an element $g$ is called
a {\em random subproduct} of the sequence $(g_i)$. This notion
was introduced by \cite{BLS88} in the context of the analysis of algorithms
on permutation groups. See, e.g., 
\cite[\S 2.3]{MR1970241} for other applications.}

Let $\alpha\in \lbrack n\rbrack$ be arbitrary. Let $j$ be the largest integer such that  $g_j$ moves $\alpha$. Then $g$ moves
$\alpha$ if and only if $g'  = g_1^{r_1} \dotsc g_j^{r_j}$ moves $\alpha$. Take $r_1,r_2,\dotsc, r_{j-1}$ as given. If
$\beta = \alpha^{g_1^{r_1} \dotsc g_{j-1}^{r_{j-1}}}$ equals $\alpha$, then $g'$ moves $\alpha$ if and only if $r_j=1$; this happens
with probability $1/2$. If $\beta\ne \alpha$, then $g'$ certainly moves $\alpha$ if $r_j=0$, and thus moves $\alpha$ with probability
at least $1/2$. Thus $g$ moves $\alpha$ with probability at least $1/2$.\footnote{This argument essentially appears in \cite[\S 6.2]{BLS88} (without proof).
It appears again, with proof and in a much more general context, 
in \cite{BCFLS}. Indeed, Lemma \ref{lem:macanus} here follows
immediately from \cite[Lem.~2.2]{BCFLS} (with $K$ equal to a point stabilizer),
and the idea of the proof of Lemma \ref{lem:macanus} given here is exactly
the same as that
 of \cite[Lem.~2.2]{BCFLS}. We thank an anonymous referee for this remark.}

Summing over all $\alpha$, we see that the expected value of the number of elements of $\lbrack n\rbrack$ moved by $g$ is at least
$n/2$. In particular, there is a $g\in A^n$ moving at least $n/2$ elements of $\lbrack n\rbrack$. 
\end{proof}

The following is the simplest sphere-packing lower bound, applied to the Hamming
distance. (The {\em Hamming distance} on $\{0,1\}^k$ is 
$d(\vec{x},\vec{y}) = |\{1\leq j\leq k: x_j \ne y_j\}|$.)
\begin{lem}
\label{sphere packing}
Let $n>0$, $k \geq 4.404 \log_2 n$, $\rho>1$. 
Let $U=\{ 0,1 \}^{k}$ the set of 
$\{0,1\}$-sequences of length $k$. Then
 there exists $V \subseteq U$, $|V| > n$ such that any two 
sequences in $V$ differ in more than $\log_2 n$ coordinates.
\end{lem}
\begin{proof}
In general, for $U$ a metric space and $V\subseteq U$ maximal with respect to 
the property that the distance between any two points of $V$ is greater than 
$r$, the closed balls of radius $r$ around the points of $V$ clearly cover $U$; 
hence, if the notion of volume is well-defined, $|V|$ is at most
$\Vol(U)$ divided by the volume of a closed ball of radius $r$. Applying
this to the Hamming distance, we obtain that, for $V\subseteq U$ maximal,
\[|V| \geq \frac{2^k}{\sum_{j=0}^{\lfloor r\rfloor} \binom{k}{j}}.\]

By, e.g., \cite[\S 10.11, Lem. 8]{MR0465509},
\[\sum_{j=0}^{\lfloor r\rfloor } \binom{k}{j} \leq 2^{k H(\lfloor r\rfloor/k)} \leq
2^{k H(r/k)}\]
for $0\leq r\leq k/2$,
where $H(x) = -x \log_2 x - (1-x) \log_2(1-x)$ is the binary entropy function.
Let $r = \log_2 n$. 
It is easy to check that, for $0\leq 
\rho\leq 1/4.404$, $1-H(\rho) > \rho$. Hence
\[|V| \geq 2^{k (1-H(r/k))} > 2^{k \cdot r/k} = 2^r = n.\]

\end{proof}

The following lemma is the main step toward answering the question raised at the beginning of the subsection. Most of the proof goes to show that, for some 
$g\in A^n$, $h\in A^\ell$ and a random $\beta\in \lbrack n\rbrack$, the orbit
of $\beta$ under $\langle g,h\rangle$ is rather large. The following is a brief
sketch. If $\beta$ were being acted upon by many random elements of 
$\Sym(\lbrack n\rbrack)$ in succession,
it would indeed traverse many points. Now think of this obvious remark as being
strengthened twice. First, let $g$ have
 large support and let $h$ be a random element of $\Sym(\lbrack n\rbrack)$.
If we let $h$ act on $\beta$ and then let $g$ act (or not) on $\beta^h$, 
and we let this happen over and over, the effect is a great deal as if
$\beta$ were being acted upon by random elements in succession: if $\beta$
has arrived at a point $x$ where it has not been before, then the random element
$h$ acts on it in a way that, as far as we are concerned, is essentially random,
in that it is almost independent of any of the parts of $h$ we have seen so far.
This makes the action of the fixed element $g$ on $x^h$ itself random.
Here comes the second strengthening: 
it is actually enough for $h$ to be the outcome of a random walk of
moderate length $\ell\leq n^{O(\log n)}$: as we know (Lemma \ref{l:ktuples}),
such an $h$ pretends to be
 a random element of $\Sym(\lbrack n\rbrack)$ very ably as far as its action on 
$k$ tuples, $k\ll \log n$, is concerned; these are all the tuples that
we have to deal with, since the above argument gives us large orbits after
$O(\log n)$ steps. 

The proof below is just a detailed and rigourous version of this sketch.

\begin{lem}\label{lem:cannibal}
Let $A\subseteq \Sym([n])$ with $A = A^{-1}$, $e\in A$ and $\langle A\rangle =\Sym([n])$ or $\Alt([n])$. Then there are $g\in A^n$, $h\in A^{\lfloor n^{27 \log n}\rfloor}$ 
such that the action of
$\langle g,h\rangle$ on $\lbrack n\rbrack$ has 
at most $175 (\log n)^2$ orbits, 
provided that $n$ is larger than an absolute constant.
\end{lem}

\begin{proof}
We will show that, for some $g\in A^n$, for
 $h\in A^{\ell}$ ($\ell \leq \lfloor n^{27 \log n}\rfloor$)
 taken randomly in a sense we will specify, and for any $\beta\in \lbrack n\rbrack$,
the expected value of $1/|\beta^{\langle g,h\rangle}|$ is at most $175 (\log n)^2/n$. (Here $\beta^{\langle g,h\rangle}$ denotes the orbit of $\beta$ under
the action of $\langle g,h\rangle \leq \Sym(\lbrack n\rbrack)$.) Now, 
$\sum_{\beta \in \lbrack n\rbrack} 1/|\beta^{\langle g,h\rangle}|$ 
is just the number of orbits of $\langle g,h\rangle$ (since
each such orbit contributes 
$|\beta^{\langle g,h\rangle}| \cdot 1/|\beta^{\langle g,h\rangle}| = 1$ to the 
sum). Hence, by the additivity of expected values,
\[\mathbb{E} (\text{number of orbits of $\langle g,h\rangle$}) =
\sum_{\beta \in \lbrack n\rbrack} \mathbb{E}\left(
\frac{1}{|\beta^{\langle g,h\rangle}|}\right) 
\leq
175 (\log n)^2 .\]
In particular, this will imply that there exists an
 $h\in A^{\ell}$ such that the number of orbits
of $\langle g,h\rangle$ is at most $175 (\log n)^2$, and so we will be done.

Let $k = \lceil 4.404 \log_2 n\rceil$. By Lemma \ref{lem:macanus},
there is
an element $g\in A^n$ with \[|\supp(g)| = \alpha n \geq n/2.\]  Let 
$\varepsilon = 1/n$ and 
$\ell = \lceil 2 n^{6 k} \log(n^{2 k}/\varepsilon)\rceil$. 
(It is easy to check that, for $n$ larger than an absolute constant, 
$\ell \leq \lfloor n^{27 \log n}\rfloor$.)
Let $h\in A^{\ell}$
be the outcome
of a random walk of length $\ell$ as in Lemma~\ref{l:ktuples}. 

Consider all words of the form 
\[ f(\vec{a})=h g^{a_1} h g^{a_2} \dotsc h g^{a_k},\]
where $\vec{a} =(a_i :1\leq i\leq k)$ runs through all sequences in
$U=\{ 0,1 \}^k$. 
For $\beta \in [n]$, we wish to estimate $|\beta^{\langle g,h\rangle}|$
from below
by counting the number of different images  
$f_\beta(\vec{a}) := \beta^{f(\vec{a})}$, for $\vec{a} \in U$.

To this end, for fixed elements $\vec{a}=(a_1,\ldots,a_k)$ and 
$\overrightarrow{a'}=(a'_1,\ldots,a'_k)$ in $U$ and $\beta \in [n]$,
we wish to bound from above the probability that 
$f_\beta(\vec{a})=f_\beta(\overrightarrow{a'})$.
We will do this by examining all possible trajectories 
$(\beta_1,\dotsc,\beta_k)$, $(\beta_1',\dotsc,\beta_k')$,
where
\[\beta_1 = \beta^{h g^{a_1}},\; \beta_2 = \beta_1^{h g^{a_2}},\; \dots,\; \beta_k = \beta_{k-1}^{h g^{a_k}}\;\;\;\; \text{and}\;\;\;\;
\beta_1' = \beta^{h g^{a_1'}},\; \dots,\; \beta_k' = \beta_{k-1}^{h g^{a_k'}},\] 
counting how many
satisfy $\beta_k \ne \beta_k'$, and then estimating the probability (for $h$ chosen randomly in the manner described above) that such a pair
of trajectories be traversed following $f(\vec{a})$ and $f(\vec{a'})$.

Let $R = \{1\leq i\leq k: a_i \ne a_i'\}$; let the elements of $R$ be 
$k_1<k_2<\dotsb<k_r$, where $r = |R|$. Let $r_0\leq r$ be fixed. Let $k' = k_{r_0}$.
Consider all tuples $(\beta_1,\beta_2,\dotsc,\beta_k,\beta_{k'}',\dotsc, \beta_k')\in
\lbrack n\rbrack^{(2k-k')+1}$ such that
\begin{enumerate}
\item\label{it:propa} $\beta_1,\beta_2,\dotsc,\beta_k,\beta_{k'}',\dotsc,\beta_k'$ are distinct from each other
and from $\beta$,
\item\label{it:propb} $\beta_1^{g^{-a_1}},\beta_2^{g^{-a_2}},\dotsc,\beta_k^{g^{-a_k}}$,
$(\beta_{k'+1}')^{g^{-a'_{k'+1}}},\dotsc,
(\beta_k')^{g^{-a'_k}}$ are distinct from each other,
\item\label{it:propc} $\beta_{k_j}\notin \supp(g)$ for every $j < r_0$, but
$\beta_{k'}\in \supp(g)$,
\item\label{it:propd} $\left(\beta_{k'}'\right)^{g^{-a_{k'}'}} = \left(\beta_{k'}\right)^{g^{ - a_{k'}}}$.
\end{enumerate}
The number of such tuples is at least 
\begin{equation}
\label{nrtuples}
\left( \prod_{j=1}^{r_0-1} (n - |\supp(g)| - j) \right)
\cdot (|\supp(g)| - 1) \cdot \prod_{j = (r_0+1)}^{ 2 k - k'}
(n - (2j-1)),
\end{equation}
where we count tuples by choosing first $\beta_{k_j}\in \lbrack n\rbrack\setminus
\supp(g)$ for $1\leq j <r_0$, then $\beta_{k'}\in \supp(g)$, then the other
$\beta_i$ and $\beta_i'$. To justify the estimate on the number of choices at each stage, notice that at the $j^{\mathrm{th}}$ choice with $j\le r_0-1$ we have to make selections from $[n]\setminus \supp(g)$ so as
 to satisfy (\ref{it:propc}) while keeping them different from previous selections and from $\beta$ (to satisfy (\ref{it:propa})). Then $\beta_{k'}$ can be chosen as an arbitrary element of $\supp(g)$ different from $\beta$. At this point, (\ref{it:propb}) is still satisfied automatically. At later choices, if $\beta_i$ or $\beta_i'$ is selected at stage $j$ then enforcing (\ref{it:propa}) eliminates $j$ possibilities and enforcing (\ref{it:propb}) eliminates $j-1$, not necessarily different, possibilities. Note that \eqref{nrtuples} also gives a valid lower estimate (namely, $0$) in the case when $r_0-1 \ge n-|\supp(g)|>0$ (the negative terms in the first product
in \eqref{nrtuples} are made harmless by a term equal to $0$). 

By Lemma~\ref{l:ktuples} (with $2k-k'$ instead of $k$, and with properties (\ref{it:propa}),
(\ref{it:propb}) as inputs), the probability that a random $h \in A^{\ell}$ satisfies
\begin{equation}\label{eq:lava}
(\beta,\beta_1,\dotsc,\beta_{k-1},\beta_{k'}',\dotsc,\beta_{k-1}')^h  = 
(\beta_1^{g^{-a_1}},\beta_2^{g^{-a_2}}, 
\beta_k^{g^{-a_k}},(\beta_{k'+1}')^{g^{-a'_{k'+1}}},\dotsc,
(\beta_k')^{g^{-a'_k}})\end{equation}
is at least $(1-\varepsilon) \frac{(n-(2k-k'))!}{n!} > (1-\varepsilon)\frac{1}{n^{2k-k'}}$. If $h$ satisfies \eqref{eq:lava} then 
$\beta^{h g^{a_1}} = \beta_1$, $\beta_1^{h g^{a_2}} = \beta_2$,\dots , $\beta_{k-1}^{h g^{a_k}} = \beta_k$. 
By properties (\ref{it:propc}) and (\ref{it:propd}), we also have 
$\beta^{h g^{a_1'}} = \beta_1$, $\beta_1^{h g^{a_2'}} = \beta_2$,\dots , $\left(\beta_{k'-1}\right)^{h g^{a_{k'}'}} 
= \left(\beta_{k'}\right)^{g^{-a_{k'}}g^{a_{k'}'}}=
\left(\beta'_{k'}\right)^{g^{-a'_{k'}}g^{a_{k'}'}} = \beta_{k'}'$; 
by (\ref{eq:lava}), we also have
$\left(\beta_{k'}'\right)^{h g^{a_{k'+1}'}} = \beta_{k' + 1}'$, \dots,
$\left(\beta_{k-1}'\right)^{h g^{a_{k}'}} = \beta_k'$. Thus, in particular, any two distinct
tuples \[(\beta_1,\beta_2,\dotsc,\beta_k,\beta_{k'}',\dotsc, \beta_k')\] give us mutually
exclusive events, even for different values of $r_0$. 
Note also that, by property (\ref{it:propa})
and what we have just said, $f_\beta(\vec{a}) = 
\beta_k \ne \beta_k' = f_\beta(\vec{a'})$.

Hence the probability $P$ that $f_\beta(\vec{a})  \ne  f_\beta(\vec{a'})$
is at least
\begin{eqnarray}\label{eq:aa}
P \ge \sum_{r_0=1}^r  \frac{1-\varepsilon}{n^{2k-k_{r_0}}} &\cdot
\left(\prod_{j=1}^{r_0-1} (n - \alpha n - j) \right)
\cdot (\alpha n - 1) \cdot \prod_{j = (r_0+1)}^{ 2 k - k_{r_0}}
(n - (2j-1)) \notag \\
&> \sum_{r_0=1}^r  \left(1 - \frac{1}{n}\right) 
\left(1 - \frac{4k}{n}\right)^{2k} 
\left(\alpha -\frac{1}{n}\right) \cdot \prod_{j=1}^{r_0-1} 
\left(1-\alpha  - \frac{j}{n} \right). 
\end{eqnarray}
If $\alpha n =|\supp(g)| \ge n- k$ then we estimate $P$ from below by the summand $r_0=1$ in \eqref{eq:aa}, yielding 
$$P > \left(1 - \frac{1}{n}\right) 
\left(1 - \frac{4k}{n}\right)^{2k} 
\left(\alpha -\frac{1}{n} \right)> 1-\frac{1}{n} - \frac{8k^2}{n} -\frac{k+1}{n} \geq 1-\frac{9k^2}{n}, $$
with the last inequality valid for $n\geq 2$.

If $\alpha n =|\supp(g)| < n- k$ then, estimating the terms $(1-\alpha-j/n)$ in the last product in \eqref{eq:aa} from below by $(1-\alpha-k/n)$, we obtain
\begin{eqnarray}
P &>& \left(1 - \frac{1}{n}\right) 
\left(1 - \frac{4k}{n}\right)^{2k} \left(\alpha -\frac{1}{n}\right) \sum_{r_0=1}^r \left(1-\alpha-\frac{k}{n}\right)^{r_0-1} \notag \\
&>&\left(1 - \frac{4k}{n}\right)^{2k+1} 
\left(\alpha -\frac{1}{n}\right) \frac{1-(1-\alpha-(k/n))^r}{(1-(1-\alpha-(k/n))} \label{eq:bb} \\
&=& \left(1 - \frac{4k}{n}\right)^{2k+1} \frac{\alpha-(1/n)}{\alpha+(k/n)}
\left( 1-(1-\alpha-(k/n))^r\right). \notag
\end{eqnarray}
Since $\alpha \ge 1/2$, we have $\frac{\alpha-(1/n)}{\alpha+(k/n)} \ge
1-\frac{2(k+1)}{n}$ and $(1-\alpha-(k/n))^r < (1/2)^r$, implying
$$P > 1- \frac{4k(2k+1)}{n} -\frac{2(k+1)}{n} -\frac{1}{2^r} > 1-\frac{9k^2}{n}
 -\frac{1}{2^r}$$
if $n\geq 3$ (since then
$k\geq 7$).

We conclude that, for any two non-identical tuples
\[\vec{a}=(a_1,\dotsc,a_k)\in \{0,1\}^k,\;\;\;\; \vec{a'}=(a_1',\dotsc, a_k')\in \{0,1\}^k\] and for 
any $\beta \in [n]$,
\[\Prob(\beta^{h g^{a_1} h g^{a_2} \dotsc h g^{a_k}}
 = \beta^{h g^{a_1'} h g^{a_2'} \dotsc h g^{a_k'}})
< \frac{9 k^2}{n} + \frac{1}{2^{d(\vec{a},\vec{a'})}} ,\]
where $d(\vec{a},\vec{a'})$ is the Hamming distance between $\vec{a}$ and
$\vec{a}'$, i.e., 
 the number of indices $1\leq j\leq k$ for which $a_j \ne a_j'$.

By Lemma~\ref{sphere packing}, there exists a set $V$ of more than $n$ tuples so that any two tuples differ in more than $\log_2 n$ coordinates. For fixed $\beta \in [n]$, writing $f_\beta(\vec{a}) = 
\beta^{h g^{a_1} h g^{a_2} \dotsc h g^{a_k}}$, $\vec{a}\in V$ for the random variable 
$\beta \mapsto f_\beta(\vec{a})$ defined using a random $h \in A^{\ell}$, we obtain that
\begin{eqnarray*}
\mathbb{E}(|\{(\vec{a},\vec{a}')\in V^2: f_\beta(\vec{a}) = f_\beta(\vec{a}')\}|)
= \sum_{\vec{a},\vec{a}'\in V} \Prob(f_\beta(\vec{a}) = f_\beta(\vec{a}')) \\ 
\leq
|V| + \left( \frac{9 k^2}{n} + \frac{1}{2^{d(\vec{a},\vec{a'})}} \right) |V|(|V|-1) <
\frac{|V|^2}{n} + \left(\frac{9k^2}{n} + \frac{1}{n}\right) |V|^2 \\ <  (9 k^2 + 2) \frac{|V|^2}{n} <
175 (\log n)^2 \frac{|V|^2}{n}\end{eqnarray*}
for $n$ larger than an absolute constant.

Concerning the length of the orbit $\beta^{\langle g,h \rangle}$, we have 
\[\begin{aligned}
\mathbb{E}\left(\frac{1}{|\beta^{\langle g,h\rangle}|}\right) &\leq
\mathbb{E}\left(\frac{1}{|\{f_\beta(\vec{a}): \vec{a}\in V\}|}\right)\\
 &\leq
\mathbb{E}\left(
\frac{|\{(\vec{a},\vec{a}')\in V^2: f_\beta(\vec{a}) = f_\beta(\vec{a}')\}|}{|V|^2}
\right)
\leq \frac{175 (\log n)^2}{n},\end{aligned}\]
where we use Cauchy-Schwarz
in the second step for the numbers $m_i$ that measure how many times a particular value $\gamma_i$ occurs among the 
$f_\beta(\vec{a})$, for some $\vec{a}\in V$. 
\end{proof}

\begin{prop}\label{prop:hast}
Let $A\subseteq \Sym([n])$ with $A = A^{-1}$, $e\in A$ and $\langle A\rangle =
\Sym([n])$ or $\Alt([n])$. If $n$ is larger than an absolute constant,
then there are $g_1, g_2, g_3\in A^{\lfloor n^{27 \log n}\rfloor}$
such that $\langle g_1, g_2, g_3\rangle$ is transitive.
\end{prop}
\begin{proof}
Let $g$, $h$ be as in Lemma \ref{lem:cannibal}.
Let $\varepsilon = 1/n^2$, $\ell = \lceil 2 n^6 \log(n^2/\varepsilon)\rceil$.
Let $g'\in A^{\ell}$ be the outcome of a random walk
of length $\ell$ as in Lemma~\ref{l:ktuples}.
Note that $\ell \leq \lfloor n^{27 \log n}\rfloor$ for $n$ larger than
an absolute constant.

Let $\Delta$ be the union of orbits of $\langle g,h\rangle$ of length less than
$\sqrt{n}$. 
Since, by Lemma \ref{lem:cannibal},
there are at most $175 (\log n)^2$ orbits of $\langle g,h\rangle$, 
we have $|\Delta| < 175 \sqrt{n}(\log n)^2$.
Let $S$ be
a set consisting of one element $\alpha$ of each orbit of length less than 
$\sqrt{n}$. Then, for
each $\alpha\in S$, Lemma~\ref{l:ktuples} implies that 
\[\Prob\left(\alpha^{g'}\in \Delta\right) \le (1+\varepsilon) \frac{|\Delta|}{n} <
\left(1+\frac{1}{n^2} \right)\frac{175(\log n)^2}{\sqrt{n}}\]
and so
\begin{equation}
\label{firstbad}
\Prob\left( (\exists {\alpha}\in S) \;
(\alpha^{g'}\in \Delta)\right) < 
\left(1+\frac{1}{n^2} \right)\frac{175^2(\log n)^4}{\sqrt{n}}.
\end{equation}

 Let $\kappa$ be an orbit of
$\langle g,h\rangle$ contained in $n\setminus \Delta$; by definition,
$|\kappa|\geq \sqrt{n}$. Let $\kappa_0$ be the largest orbit; by the 
pigeonhole principle, $|\kappa_0| > n/(175 (\log n)^2)$. 
Then
\[\mathbb{E}(|\kappa^{g'}\cap \kappa_0|) = \sum_{\alpha\in \kappa}
\Prob(\alpha^{g'}\in \kappa_0) \geq \sum_{\alpha\in \kappa} (1-\varepsilon) 
\frac{|\kappa_0|}{n} = (1-\varepsilon) \frac{|\kappa| |\kappa_0|}{n},\]
whereas
\[\begin{aligned}
\mathbb{E}\left(|\kappa^{g'} \cap \kappa_0|^2\right) & =\sum_{\alpha, \beta \in \kappa} \Prob\left(\alpha^{g'} \in \kappa_0 \wedge
\beta^{g'} \in \kappa_0\right)\\ &= \sum_{\alpha\in \kappa} \Prob(\alpha^{g'} \in
\kappa_0) + \mathop{\sum_{\alpha,\beta\in\kappa}}_{\alpha\ne \beta}
\mathop{\sum_{\alpha',\beta'\in\kappa_0}}_{\alpha'\ne \beta'}
\Prob((\alpha,\beta)^{g'} = (\alpha',\beta'))\\
&\leq \sum_{\alpha\in \kappa} (1 + \varepsilon) \frac{|\kappa_0|}{n} + 
\sum_{\alpha,\beta\in \kappa, \alpha \ne \beta} (1+\varepsilon) \frac{|\kappa_0| (|\kappa_0|-1)}{
n (n-1)}\\ &\leq (1+\varepsilon) \left(\frac{|\kappa_0| |\kappa|}{n} + 
\frac{|\kappa| (|\kappa|-1) |\kappa_0| (|\kappa_0|-1)}{n(n-1)}\right)
\\&\leq (1+\varepsilon) \left(\frac{|\kappa_0| |\kappa|}{n} +
\frac{|\kappa|^2 |\kappa_0|^2}{n^2} \right).
\end{aligned}\]
Thus
\[\begin{aligned}
\Var(|\kappa^{g'}\cap \kappa_0|) &= \mathbb{E}(|\kappa^{g'}\cap \kappa_0|^2) -
\mathbb{E}(|\kappa^{g'}\cap \kappa_0|)^2\\
&\leq (1+\varepsilon) \left(\frac{|\kappa_0| |\kappa|}{n} +
\frac{|\kappa_0|^2 |\kappa|^2}{n^2}\right) - (1-\varepsilon)^2 \frac{|\kappa_0|^2
|\kappa|^2}{n^2}\\
&\leq 3 \varepsilon \frac{|\kappa|^2 |\kappa_0|^2}{n^2} + (1+\varepsilon)
\frac{|\kappa_0| |\kappa|}{n} < \left(1 + \frac{4}{n}\right)
\frac{|\kappa_0| |\kappa|}{n} .
\end{aligned}\]
By Chebyshev's inequality,
\[\begin{aligned}
\Prob(\kappa^{g'} \cap \kappa_0 = \emptyset) &\leq
\frac{\Var(|\kappa^{g'}\cap \kappa_0|)}{
\mathbb{E}(|\kappa^{g'}\cap \kappa_0|)^2}\\
&\leq \frac{(|\kappa| |\kappa_0|/n) (1 + 4/n)}{(1-\varepsilon)^2 \frac{|\kappa|^2
|\kappa_0|^2}{n^2}} \leq \frac{12 n}{|\kappa| |\kappa_0|} <
\frac{12 \cdot 175 (\log n)^2}{\sqrt{n}}.\end{aligned}\]
Hence 
\begin{equation}
\label{secondbad}
\Prob\left( (\exists {\kappa}\subseteq ([n] \setminus \Delta)) \;
(\kappa^{g'} \cap \kappa_0 = \emptyset)\right) < 
\frac{12 \cdot 175^2(\log n)^4}{\sqrt{n}}.
\end{equation}
Now, for $n$ larger than a constant,
\[
\left(1+\frac{1}{n^2} \right)\frac{175^2(\log n)^4}{\sqrt{n}}
+\frac{12 \cdot 175^2(\log n)^4}{\sqrt{n}} < 1.
\]
Therefore, \eqref{firstbad} and \eqref{secondbad} imply that with positive probability,
(a) $\kappa^{g'}$ intersects $[n] \setminus \Delta$
for every orbit $\kappa$ not contained in $[n] \setminus \Delta$ and (b) $\kappa^{g'}$ intersects $\kappa_0$ for every orbit 
$\kappa$ contained in $[n] \setminus \Delta$. In particular,
this happens for some $g' \in A^{\ell}$. Properties (a) and (b) imply that
$\langle g,h,g'\rangle$ is transitive. We set $g_1 = g$, $g_2 = h$, $g_3 = g'$
and are done.
\end{proof}

We will later use\footnote{If we wished to, we could
use it to obtain a set $S$ of generators of $\Alt(\lbrack n\rbrack)$ or
$\Sym(\lbrack n\rbrack)$ simply by setting $k=6$: the Classification
of Finite Simple Groups implies that a $6$-transitive group must be
either alternating or symmetric.} the following corollary with $k=2$. 

\begin{cor}
\label{cor0.5}
Let $A\subseteq \Sym[(n])$ with $A = A^{-1}$, $e\in A$ and $\langle A\rangle = \Sym([n])$ or $\Alt([n])$. Let $k\geq 1$. If $n$ is larger than a
constant depending only on $k$, 
then there is a set $S\subseteq A^{\lfloor n^{28 \log n}\rfloor}$ 
of size at most $3 k$ such that $\langle S\rangle$ is $k$-transitive.
\end{cor}
\begin{proof}
Let $\alpha_1 \in \lbrack n\rbrack$ be arbitrary.
 Since $\langle A\rangle$ is transitive, Lemma~\ref{lem:apeman1} implies that
$\alpha_1^{A^n} = \lbrack n\rbrack$.
 Let $G = \Sym([n])$,
$H = G_{\alpha_1}$, $A' = A^n$.
Since $\alpha_1^{A'} = \lbrack n\rbrack$, $A'$ intersects
every coset of $H$ in $G$.
By Schreier's Lemma (Lem~\ref{schreier}), it follows that
$(A')^3 \cap H$ generates $\langle A\rangle \cap H$, which is either
$\Sym(\lbrack n\rbrack \setminus \{\alpha_1 \})$ or 
$\Alt(\lbrack n\rbrack \setminus \{\alpha_1\})$. Let $A_1 = (A')^3\cap H$.

Iterating, we obtain a sequence of sets $A_0 =A, A_1, A_2,\dotsc, A_{k-1}
\subseteq \Sym([n])$ and a sequence of elements $\alpha_1,
\alpha_2, \dotsc,\alpha_{k-1}\in \lbrack n\rbrack$
such that $A_i\subseteq A_{i-1}^{3n}$ and $\langle A_i\rangle$ is $\Sym(\lbrack n\rbrack\setminus \{\alpha_1,\dotsc,\alpha_i\})$ or $\Alt(\lbrack n \rbrack \setminus \{\alpha_1,\dotsc,\alpha_i\})$.

Let $(g_1)_i$, $(g_2)_i$, $(g_3)_i$ be
 as in Prop.\ \ref{prop:hast}, applied with $A_i$ instead of $A$.
Then $(g_1)_i, (g_2)_i, (g_3)_i\in A^{(3 n)^i \lfloor n^{27 \log n}\rfloor}$ and
$\langle (g_1)_i, (g_2)_i, (g_3)_i\rangle \subseteq \Sym(\lbrack n
\rbrack \setminus \{\alpha_1,\dotsc,\alpha_i\})$ is transitive on
$\lbrack n\rbrack \setminus \{\alpha_1,\dotsc,\alpha_i\}$ for
$0\leq i\leq k-1$. Thus, for
$S = \bigcup_{i=0}^{k-1} A_i$, $\langle S\rangle$ is $k$-transitive on
$\lbrack n\rbrack$. 
\end{proof}

\section{The splitting lemma and its consequences}
\label{sec:splitting}
We will prove what is in effect an adaptation of Babai's splitting lemma
(proven for groups in \cite[Lem.\ 3.1]{Bab82}) to the case of sets. 
This is a key point in this paper: the splitting
lemma will allow us to construct long stabilizer chains with large orbits.

The following easy lemma will make an ``unfolding'' step possible.
\begin{lem}
\label{stabilizer of conjugate}
Let $A \subseteq \Sym([n])$, $\Sigma\subseteq [n]$ and $g \in \Sym([n])$. Then  
$$g A_{(\Sigma^g)} g^{-1} =(gAg^{-1})_{(\Sigma)}.$$
\end{lem}
\begin{proof}
We have $\Sym([n])_{(\Sigma^g)}=g^{-1}\Sym([n])_{(\Sigma)}g$.
Therefore, 
\begin{eqnarray*}
A_{(\Sigma^g)}&=&A \cap \Sym([n])_{(\Sigma^g)} =A \cap 
g^{-1}\Sym([n])_{(\Sigma)}g\\
&=&g^{-1}(gAg^{-1}\cap \Sym([n])_{(\Sigma)})g=g^{-1}(gAg^{-1})_{(\Sigma)}g.
\end{eqnarray*}
\end{proof}

Notice a feature of the following statement -- there is a high
power of $A$ in the assumptions, not just in the conclusion. 
We will ``unfold'' the high power of $A$ in the course of the proof.
(By $\Sigma^S$ we mean the set  $\Sigma^S = \{\alpha^g: \alpha\in \Sigma, 
g\in S\}$.)
\begin{prop}[Splitting Lemma]
\label{prop:babai}
Let $A\subseteq \Sym([n])$ with $A=A^{-1}$, $e\in A$ and $\langle A\rangle$ 
 $2$-transitive. 
Let $\Sigma\subseteq \lbrack n\rbrack$. Assume that
there are at least $\rho n(n-1)$ ordered pairs $(\alpha,\beta)$ of distinct elements of 
$[n]$ such that there is no 
$g\in (A^{\lfloor 9n^6\log n\rfloor})_{(\Sigma)}
$ with $\alpha^g=\beta$. Then there is a subset $S$
 of $A^{\lfloor 5n^6\log n\rfloor}$ with $$(A A^{-1})_{(\Sigma^S)}=\{e\}$$  and
$$|S|\leq \left\lceil \frac{2}{\log(3/(3-2\rho))} \cdot \log n \right\rceil.$$
\end{prop}

\begin{proof}
Set $\ell = \lceil 2n^6\log(n^2/(1/3))\rceil$; note that
$\ell \leq \lfloor 5 n^6 \log n\rfloor$ and $2\ell+2 \leq \lfloor 9 n^6 \log n\rfloor$ for $n\geq 5$.
(For $n<5$, the statement is trivial.)
By Lemma~\ref{l:ktuples} applied with $k=2$ and $\varepsilon=1/3$, we obtain that 
given any two distinct elements $\alpha,\beta\in [n]$ and $g\in A^\ell$, the pair 
$(\alpha^g,\beta^g)$ adopts any possible value $(\alpha',\beta')$ with probability at least 
$(1-1/3)/(n(n-1))$, where we choose $g\in A^{\ell}$ with the distribution in
Lemma~\ref{l:ktuples} ($g = g_1 g_2\dotsb g_\ell$, $g_i$ chosen
independently from $A'\cup \{e\}$, where $A'$ is a symmetric 
subset of $A$). Since this distribution is symmetric,
this is the same as saying that 
$(\alpha^{g^{-1}},\beta^{g^{-1}})$ 
adopts any possible value $(\alpha',\beta')$ with probability at least 
$(1-1/3)/(n(n-1))$.

Now, given $(\alpha,\beta)$ and $g\in A^\ell$, we have $h\in (A A^{-1})_{(\Sigma^g)}$ and 
$\alpha^h=\beta$ if and only if $ghg^{-1}\in g(A A^{-1})_{(\Sigma^g)}g^{-1}$ and 
$(\alpha^{g^{-1}})^{ghg^{-1}}=\beta^{g^{-1}}$. By Lemma~\ref{stabilizer of conjugate} applied 
to $A A^{-1}$, we have that $g h g^{-1}\in g (A A^{-1})_{(\Sigma^g)} g^{-1}$ 
only if 
$g h g^{-1}\in (g A A^{-1} g^{-1})_{(\Sigma)}$, which in turn can happen only if 
$ghg^{-1}\in (A^{2\ell+2})_{(\Sigma)}$. Thus, if there is no element 
$j\in (A^{2\ell+2})_{(\Sigma)}$ with $\alpha^{g^{-1}j}=\beta^{g^{-1}}$, then there is no element 
$h\in (A A^{-1})_{(\Sigma^g)}$ with $\alpha^h=\beta$. (This is
the ``unfolding'' step we referred to before.)

Since by hypothesis there are at least $\rho n(n-1)$ 
ordered pairs $(\alpha',\beta')$ such
that there is no element $j\in (A^{2\ell+2})_{(\Sigma)}$ with 
$\alpha'^j=\beta'$, and since 
$(\alpha^{g^{-1}},\beta^{g^{-1}})$ equals any such pair with probability at least 
$(2/3)/(n(n-1))$, we see that the probability that there is no element 
$h\in (A A^{-1})_{(\Sigma^g)}$ with $\alpha^h=\beta$ is at least $2\rho/3$.

Let $S$ be a set of $r$ random $g\in A^\ell$ (chosen independently, with the
 distribution as above). The probability that for every $g\in S$ there is an 
element $h\in (A A^{-1})_{(\Sigma^g)}$ with $\alpha^h=\beta$ is at most 
$(1-2\rho/3)^r$. This 
must happen if there is an element $h\in (A A^{-1})_{\Sigma^S}$ such that 
$\alpha^h=\beta$. Thus, the probability that there is such an $h$ is at most 
$(1-2\rho/3)^r$, and the probability that there is such an $h$ for at least one of 
the $n(n-1)$ pairs $(\alpha,\beta)$ is at most $n(n-1)(1-2\rho/3)^r$.

Setting $r= \lceil (\log n^2)/(\log 3/(3-2\rho))\rceil$, 
we obtain that the probability that there is such an $h$ for at least one pair is less than $1$. 
Hence there is a set $S\subseteq A^{\ell}$ with at most $r$ elements such that, for every pair $(\alpha,\beta)$ 
of distinct elements of $[n]$, there is no $h\in (A A^{-1})_{(\Sigma^S)}$ with $\alpha^h=\beta$. 
This implies immediately that the only element of $(A A^{-1})_{(\Sigma^S)}$ is the identity.
\end{proof}

\begin{cor}\label{cor:glub}
Let $A\subseteq \Sym([n])$ with $A=A^{-1}$, $e\in A$ and $\langle A\rangle$ 
 $2$-transitive. Let $A' = A^{\lfloor 9n^6\log n\rfloor}$.
Let $\Sigma\subseteq \lbrack n\rbrack$ be such that 
\[|\alpha^{A'_{(\Sigma)}}| < (1-\rho) n\]
for every $\alpha\in \lbrack n\rbrack$, where $\rho\in (0,1)$. Then
\[|\Sigma| > \frac{\log |A|}{
\left\lceil \frac{2}{\log (3/(3-2\rho))}
\cdot \log n\right\rceil \cdot \log n}.\]
In particular, if
$\rho=0.05$ then $|\Sigma| > (\log |A|)/(60 (\log n)^2)$.
\end{cor}
\begin{proof}
Since $|\alpha^{A'_{(\Sigma)}}| < (1-\rho)n$ for every $\alpha\in [n]$, there are at least
$\rho n (n-1)$ pairs $(\alpha,\beta)$ such that there is no 
$g\in A'_{(\Sigma)}$ with $\alpha^g = \beta$.
By Prop.\ \ref{prop:babai}, there is a set $S\subseteq \Sym([n])$ such that
$(A A^{-1})_{(\Sigma^S)} = \{e\}$ and $|S|\leq
\left\lceil \frac{2}{\log (3/(3-2\rho))}\cdot \log n\right\rceil 
$. Since $(A A^{-1})_{(\Sigma^S)} = \{e\}$, we know, 
by Lemma \ref{lem:wort}, 
that $|\Sigma^S| \geq \log_n |A|$. Clearly $|\Sigma^S|\leq |S|
|\Sigma|$. Hence
\[|\Sigma|\geq \frac{\log_n |A|}{|S|}\geq \frac{\log |A|}{
\left\lceil \frac{2}{\log (3/(3-2\rho))}
\cdot \log n\right\rceil \cdot \log n} .\]
\end{proof}

A key idea in the proof of the Main Theorem is the following. For $A \subseteq \Sym([n])$, we can construct $A' = A^{\lfloor 5n^6\log n\rfloor}$ and a set 
$\Sigma = \{\alpha_1,\alpha_2,\dotsc\} \subseteq [n]$ starting with an empty set
and taking at each step $\alpha_i$ to be an element such that
$|\alpha_i^{(A')_{(\alpha_1,\dotsc,\alpha_{i-1})}}| \ge (1-\rho)
n$ (say); if no such element exists,
we stop the procedure. By Cor.\ \ref{cor:glub}, 
$|\Sigma|$ must be large.

An application of Lemma \ref{lem:brioche} will give that, for
$A'' = (A')^{16 n^6}$, the set $(A'')_{\Sigma}$ contains a copy of
$\Alt(\Delta)$, where $\Delta \subseteq \Sigma$ and $|\Delta|\geq (1-\rho) |\Sigma|$. Such a large alternating group certainly looks like a valuable tool.

\section{Proof of the main theorem}\label{sec:proof of main}

The core of this section is Proposition \ref{growth}.
It is a growth result, but not
quite of type $|A \cdot A \cdot A|\geq |A|^{1+\varepsilon}$ or $|A^k|\geq |A|^{1+\varepsilon}$.
What will grow by a factor at each step is not the number of elements
$|A|$ of $A$, but rather the length $m$ of a sequence $\alpha_1,\ldots,
\alpha_m$ such that the orbits \begin{equation}\label{eq:kost}
\alpha_1^A, \alpha_2^{A_{\alpha_1}},
\alpha_3^{A_{(\alpha_1,\alpha_2)}},\ldots,\alpha_m^{A_{(\alpha_1,\alpha_2,
\dotsc, \alpha_{m-1})}}\end{equation} are all large.

This growth result (Prop.~\ref{growth}) will be applied iteratively.
There are two ways for the iteration to stop: (a) an element we
construct could fix a large set pointwise (we call this the case of 
{\em exit}), or (b) a group we work with
could fail to have a large alternating composition factor. In case (a), 
we obtain all of $G=\Alt(\lbrack n\rbrack)$ in a few steps by
Thm.~\ref{bbssmallsupport}. In case (b), we can {\em descend} 
to the problem of proving small diameter for $n'$ smaller
than $n$ by a constant factor. (Here, as in ``infinite descent'', the
term ``descent'' means the same as induction, seen backwards.)

\begin{center}
* * *
\end{center}

Let us sketch briefly the proof of Prop.~\ref{growth}.
First, we use (\ref{eq:chrosh}) to construct many elements in the
setwise stabilizer $G_{\Sigma}$, where $\Sigma = \{\alpha_1,\ldots,\alpha_m\}$;
in fact we get an entire copy of a large alternating group in 
$(G_{\Sigma})|_\Sigma$ (Lemma \ref{lem:brioche}). This is the {\em setup}.
Then comes the {\em
creation} step: we use the action by conjugation of $G_{\Sigma}$ on the pointwise
stabilizer $G_{(\Sigma)}$ to construct many elements of $G_{(\Sigma)}$
(Lemma~\ref{lem:tomev}). We {\em organise} these new elements
(all in a power $A'$ of $A$) as follows: we 
apply Cor.~\ref{cor:glub} (a consequence
of the splitting lemma) to lengthen our stabilizer
chain $A' \supseteq A'_{\alpha_1} \supseteq \dotsc \supseteq
A'_{(\alpha_1,\dotsc,\alpha_{m})} \supseteq \dotsc$
up to $A'_{(\alpha_1,\dotsc,\alpha_{m+\ell})}$
in such a way that the orbits (defined as in (\ref{eq:kost})) are still large.
We repeat the {\em organiser} step
 about $\gg (\log n)/(\log m)$ times. There are only
two ways for this procedure to stop prematurely, namely, {\em exit} and
{\em descent} (cases (a) and (b) discussed above).

\begin{center}
* * *
\end{center}


We start by proving the lemma containing the {\em creation} step:
we give a way to construct many elements in a subgroup $H^-$ of a group 
$G$. The basic idea
is the application of the orbit-stabilizer principle to the action by
conjugation of a subgroup $H^+ \le N_G(H^-)$ on $H^-$, where 
$N_G(H^-)$ is the normaliser of $H^-$.
\begin{lem}\label{lem:tomev}
Let $G=\Sym([n])$ or $\Alt([n])$, $H^- \le G$, $H^+ \le N_G(H^-)$,
$\Gamma$ an orbit of both $H^-$ and $H^+$. Let $Y = \{y_1,\dots,y_r\}\subseteq H^-$
be such that $\langle Y\rangle|_\Gamma$ is $2$-transitive on $\Gamma$. 
Let $B\subseteq H^+$. Then either
\begin{enumerate}
\item\label{it:gara}
 there is a $b\in B B^{-1} \setminus \{e\}$ fixing $\Gamma$ pointwise, or
\item\label{it:garb} $|B^{-1} Y B \cap H^-|\geq |B|^{1/r}$.
\end{enumerate}
\end{lem}
\begin{proof}
Consider the action of $B$ on $\vec{y} = (y_1,\dotsc,y_r)$ by conjugation: for
$b\in B$, we define $\vec{y}^b := (y_1^b,\dotsc,y_r^b)$, where 
$y^b = b^{-1} y b$. Assume first that there are 
two distinct elements $b_1, b_2\in B$ such that
$\vec{y}^{b_1}|_\Gamma = \vec{y}^{b_2}|_{\Gamma}$. Then $b_1 b_2^{-1}|_\Gamma$
centralizes $\vec{y}|_\Gamma$, implying that $b_1 b_2^{-1}|_\Gamma 
\in C(\langle Y\rangle|_\Gamma) = \{e\}$. 
(As is well-known and can be easily seen, the centralizer of a doubly
transitive group, such as $\langle Y\rangle|_\Gamma < \Sym(\Gamma)$, is trivial.)
Hence $b_1 b_2^{-1}\in B$ fixes
$\Gamma$ pointwise without being the identity, i.e., conclusion 
(\ref{it:gara}) holds.

Assume now that the restrictions $\vec{y}^b|_\Gamma$ are all distinct.
Hence, by the pigeonhole principle, there exists an index $j\in \{1,\dotsc,r\}$
such that the set $W$ of conjugates of $y_j$ by $B$ satisfies $|W|_\Gamma|
\geq |B|^{1/r}$. Observe that all elements of $W$ are in $H^-$, as $Y\subset H^-$
and $B\subset N_G(H^-)$. Hence $|B^{-1} Y B \cap H^-|\geq |W| \geq |B|^{1/r}$. 
\end{proof}

The following useful lemma is in part an easy application of Schreier's lemma
and in part a consequence of a trick based on the following trivial fact: 
one clearly cannot 
have two disjoint copies within $\lbrack n\rbrack$ of an orbit of size greater than $n/2$.

\begin{lem}\label{lem:sonofschreier}
Let $\Delta\subseteq \lbrack n\rbrack$. Let $B^+\subseteq (\Sym(n))_\Delta$ with
$B^+ = (B^+)^{-1}$, $e\in B^+$. 
Assume $B^+|_\Delta$ is $\Alt(\Delta)$ or $\Sym(\Delta)$. Let $B^- = \left((B^+)^3\right)_{(\Delta)}$. 

Then $\langle B^-\rangle = \langle B^+\rangle_{(\Delta)} \lhd \langle B^+
\rangle$.
Furthermore, if  $\langle B^-\rangle$ 
has an orbit $\Gamma$ of length greater than $n/2$,
then $\Gamma$ is also an orbit of $\langle B^+\rangle$.
\end{lem}
\begin{proof}
Since $B^+|_\Delta$ is a group ($\Alt(\Delta)$ or $\Sym(\Delta)$), $B^+|_\Delta = 
\langle B^+ \rangle|_\Delta$. Thus $B^+$ contains an element from every coset of
$\langle B^+\rangle_{(\Delta)}$ in $\langle B^+\rangle$
and so, by Lemma \ref{schreier},
$B^-$ contains a set of generators of 
$\langle B^+\rangle_{(\Delta)}$. Hence $\langle B^-\rangle = 
\langle B^+\rangle_{(\Delta)}$. In particular, $\langle B^-\rangle \lhd
\langle B^+\rangle$, as $\langle B^+\rangle_{(\Delta)}$ is the
 kernel of the action of $\langle B^+\rangle$ on $\Delta$.

The orbits of the normal subgroup $\langle B^-\rangle \lhd
\langle B^+\rangle$ are blocks of imprimitivity for $\langle B^+\rangle$.
Since one cannot have two blocks of length greater than $n/2$, 
$\langle B^+\rangle$ leaves $\Gamma$ invariant as a set, and so
$\Gamma$ is an orbit of $\langle B^+\rangle$.
\end{proof}

The following lemma is also crucial to the descent step.
In the proof of the lemma, we use Lemma \ref{lem:small support}
to guarantee the existence of an element that we then construct
by other means.

\begin{lem}\label{lem:cases3a}
Let $G = \Sym([n])$ or $\Alt([n])$. Let $\Delta\subseteq \lbrack n\rbrack$,
$|\Delta|\geq (\log n)^2$.
Let $A\subseteq G$ with $A = A^{-1}$, $e\in A$ and $\langle A\rangle = G$. 
Let $B^+\subseteq (A^l)_\Delta$, $l\geq 1$, with 
$B^+ = (B^+)^{-1}$, $e\in B^+$. 
Assume $B^+|_\Delta$ is $\Alt(\Delta)$ or $\Sym(\Delta)$.
Let $B^- = \left((B^+)^3\right)_{(\Delta)}$. 
Assume $\langle B^-\rangle$ has an orbit $\Gamma$
of length at least
$\rho n$, for some $\rho > 8/9$.

If all alternating composition factors
$\Alt(k)$ of $\langle B^-\rangle$ satisfy $k\leq \delta n$,
where $\delta>0$, and
\begin{equation}\label{eq:soproni}
\max_{k\leq \delta n} \diam(\Alt(k)) \leq D_{\delta},
\end{equation}
for some $D_{\delta}>0$, and $n$ is larger than an absolute constant,
then 
\[A^{\lfloor l e^{c (\log n)^3} \cdot D_{\delta}\rfloor 
} \supseteq \Alt(\lbrack n\rbrack),\]
where $c=c(\rho)$ depends only on $\rho$.
\end{lem}
\begin{proof}
The group $U:=\langle B^- \rangle|_{\Gamma}$ is transitive. It is also
 isomorphic to a quotient of $\langle B^- \rangle$, so $U$ also has no
alternating composition factors $\Alt(k)$ with $k>\delta n$.
By Thm.~\ref{trandiameter} and by (\ref{eq:soproni}),
there exists an absolute constant $C_1$ such that for 
\begin{equation}
\label{eq:bs11}
u:=\lfloor e^{C_1(\log n)^3}\cdot D_{\delta}\rfloor, \ \ \ (B^-)^u|_{\Gamma}=U.
\end{equation}
Let $H = \langle B^+\rangle$. By Lemma \ref{lem:sonofschreier},
$\Gamma$ is an orbit of $H$.
If $n$ is large enough that 
Lemma~\ref{lem:small support} applies then there exists 
a non-identity element $g \in H$
of support less than $|\Gamma|/4$ on $\Gamma$. 
Take $h \in B^+$ with $h|_{\Delta}=g|_{\Delta}$. Then 
$gh^{-1} \in \langle B^+\rangle_{(\Delta)} = \langle B^- \rangle$ and so,
by \eqref{eq:bs11},
there exists $b \in (B^-)^u$ with $gh^{-1}|_{\Gamma}=b|_{\Gamma}$. 
Therefore, $bh \in (B^+)^{3 u+1}$ satisfies $bh|_{\Gamma}=g|_{\Gamma}$.
Since $g$ fixes at least $(3/4)|\Gamma|\geq (3/4)\cdot \rho n > (2/3) n$
points in $\Gamma$, 
we have $|\supp(bh)|\le (1- (3/4) \rho) n < n/3$. 
By Thm.~\ref{bbssmallsupport}, $(A \cup \{ bh, (bh)^{-1} \})^{K n^8}$
contains $\Alt([n])$, where $K=K(\varepsilon)$
($\varepsilon =1- (3/4) \rho <1/3$)
 is the number defined in Thm.~\ref{bbssmallsupport}.
Since $A \cup \{ bh, (bh)^{-1} \} \subseteq A^{(3 u + 1) l}$, we are done.
\end{proof}

We come to the key results in the paper. They will be given as two 
separate propositions, proved by a back-and-forth inductive process.
For the sake of clarity, we will state them in terms of functions $F_1,
F_2:\mathbb{R}^+\to \mathbb{R}^+$ obeying certain relations; we will
later specify functions satisfying these relations.

\begin{prop}
\label{growth}
Let $G = \Sym([n])$ or $\Alt([n])$. Let $A\subset G$ with $A = A^{-1}$,
$e\in A$, and $\langle A\rangle = G$. Let $\alpha_1,
\alpha_2,\dotsc,\alpha_{m+1}\in \lbrack n\rbrack$ be such that
\begin{equation}\label{eq:chrosh}
\left| \alpha_i^{A_{(\alpha_1,\dotsc,\alpha_{i-1})}}\right|\geq \frac{9}{10} n
\end{equation}
for every $i=1,2,\dotsc,m+1$, where $m\geq (\log n)^2$.

There are absolute constants $n_0\in \mathbb{Z}^+$ and $K,c_1,c_2,
c_3>0$ such that the following holds. Assume $n\geq n_0$. 
Assume also that Proposition \ref{main theorem} holds for all smaller
values of $n$ with respect to some 
increasing function $F_2:\mathbb{R}^+\to \mathbb{R}^+$.
Let $F_1:\mathbb{R}^+\to \mathbb{R}^+$ be such that, for all $n\in \mathbb{Z}^+$,
\begin{equation}\label{eq:peterson}\begin{aligned}
F_1(n) &\geq \max\left(n^{c_3 \log n} e^{c_1 (\log n)^3}
 F_2(0.95 n), 2 K n^{c_3 \log n + 8}\right). 
\end{aligned}\end{equation}
Then either
\begin{equation}\label{eq:jlayt}
A^{\lfloor F_1(n)\rfloor} \supseteq \Alt(\lbrack n\rbrack)
\end{equation}
or there are $\alpha_{m+2}, \alpha_{m+3},\dotsc , \alpha_{m+l+1}\in \lbrack
n\rbrack$,
$l \geq c_2 (m \log m)/(\log n)$, such that
\begin{equation}\label{eq:doublestar}
\left| \alpha_i^{A'_{(\alpha_1,\dotsc,\alpha_{i-1})}}\right|\geq \frac{9}{10} n
\end{equation}
for $A' = A^{\lfloor n^{c_3 \log n}\rfloor}$ and every
$i = 1, 2, \dotsc, m+l+1$.
\end{prop}


An easy application of Proposition~\ref{growth} proves Proposition~\ref{main theorem} (which is equivalent to our Main Theorem). Conversely, in order to prove Proposition~\ref{growth}, 
we will use Proposition~\ref{main theorem} for smaller values of $n$ in an inductive process. 
In the proofs of Prop.~\ref{growth} and Prop.~\ref{main theorem}, we assume
that $n$ is greater than a well-defined (but not explicitly computed)
absolute constant $n_0$; we take $n_0$ to be large enough to satisfy the
assumptions made in the course of both proofs.
In the statement of Prop.~\ref{growth}, the assumption is made explicitly;
in the statement of Prop.~\ref{main theorem}, the assumption
is allowed by (\ref{eq:worsi}), which
implies that, when $n\leq n_0$, the bound $\diam(\Gamma(G,Y))\leq F_2(n)$
is trivial and there is nothing to prove.

\begin{prop}
\label{main theorem}
Let $G=\Sym([n])$ or $\Alt([n])$. Let $Y\subseteq G$ with $Y = Y^{-1}$, $e\in Y$
and $G=\langle Y\rangle$. 

Assume Prop. \ref{growth} holds for $n$
with respect to some function $F_1:\mathbb{R}^+\to \mathbb{R}^+$.
Let $c_2$ and $c_3$ be the absolute constants in the statement of 
Prop.~\ref{growth}; let $n_0$ be at least as large as in Prop. \ref{growth}. 
Let $F_2:\mathbb{R}^+\to \mathbb{R}^+$ be such that
\begin{equation}\label{eq:worsi}
F_2(n) \geq \max\left(e^{(\log n)^3 + 2 \log n + c' c_3 (\log n)^3 \log \log n}
F_1(n) + 2,n_0!\right)\end{equation}
for some $c'>c_2$ and all $n\in \mathbb{Z}^+$. Then
\[\diam(\Gamma(G,Y))\leq F_2(n),\]
provided that $n_0$ is larger than a constant depending only on $c_2$ and $c'$.
\end{prop}
The proof consists just of a repeated use of Proposition~\ref{growth}, plus
some accounting.
\begin{proof}
We can assume that $n$ is large enough that $m_0\leq 0.1 n \leq n-3$ for $m_0 = 
 \lfloor (\log n)^2\rfloor + 1$ 
and so $G$ acts transitively on the set $X$ of all $(m_0+1)$-tuples. Hence, by
Lemma \ref{lem:apeman1}, the set $A_0 := Y^{n^{m_0+1}}\supseteq Y^{|X|}$ acts
transitively on the set of all $(m_0+1)$-tuples. Thus (\ref{eq:chrosh})
holds with $A_0$ instead of $A$, $m_0$ instead of $m$ and 
$\alpha_i = i$ for $i=1,2,\dotsc,m_0+1$. We apply 
Proposition~\ref{growth} with these parameters, assuming $n\geq n_0$,
where $n_0$ is the absolute constant in the statement of Prop.~\ref{growth}.
We obtain either (\ref{eq:jlayt}) or (\ref{eq:doublestar}).

In the latter case, we set $\ell_0 = \ell$, $m_1 = m_0+\ell_0$, and iterate:
we apply Proposition~\ref{growth} to
\[A_1 = A_0^r,\;\; A_2 = A_1^r = A_0^{r^2},\;\; A_3 = A_2^r = A_0^{r^3},\dotsc\]
where $r = \lfloor n^{c_3 \log n}\rfloor$. (After each step, we 
``save'' the output $\ell$ to $\ell_i$ and 
set $m_{i+1} = m_i + \ell_i$ .)
We stop when we obtain
(\ref{eq:jlayt}); say this happens when we apply Proposition~\ref{growth} 
with $A = A_k = A_0^{r^k}$. 

It remains to estimate $k$. 
By Proposition~\ref{growth},
\begin{equation}\label{eq:honeg}
m_{i+1} \geq (1 + (c_2 \log m_i)/(\log n)) \cdot m_i.\end{equation}
We want to compute how many times we have to iterate (\ref{eq:honeg}) before
we run into a contradiction with $m_i\leq n$.

For $1\leq j\leq \log n$, let $t_j$ be the largest index $i$ between $0$ and $k$
such that $m_i< e^j$; if no such index exists, set $t_j = 1$. We have
$m_0\geq 3$ and so $t_1 = 1$. By (\ref{eq:honeg}) and
$(1 + c_2 j/(\log n))^{\lfloor (\log n)/(
c_2 j)\rfloor + 2} > e$, we have
$t_{j+1} \leq t_j + \lfloor (\log n)/(c_2 j)\rfloor + 3$.
Thus \[\begin{aligned}
t_{\lfloor \log n\rfloor} +1 &\leq t_1 + 1 + \sum_{j=1}^{\lfloor \log n\rfloor -1} (t_{j+1} - t_j)
\\ &\leq 2 + \sum_{1\leq j\leq \log n} \left( \frac{\log n}{c_2 j} + 3\right)
\leq c' \log n \log \log n\end{aligned}\]
for any $c'>1/c_2$,
with the last inequality valid if $n$ is larger than a constant
depending only on $c$ and $c'$.
Since $t_{\lfloor \log n\rfloor} + 2 > k$ (because $m_k\leq n$), we get that
$k\leq c' \log n \log \log n$.

Thus
\[A_k = A_0^{r^k} \subseteq Y^{n^{\lfloor (\log n)^2\rfloor+2}\cdot r^{\lfloor
c' \log n \log \log n\rfloor}} \subseteq 
Y^{\lfloor e^{(\log n)^3 + 2 \log n + c' c_3 (\log n)^3 \log \log n}\rfloor},\]
Then, by (\ref{eq:jlayt}) (valid for $A=A_k$),
we obtain
\[\begin{aligned}
\Alt(\lbrack n\rbrack) &\subseteq 
(Y^{\lfloor e^{(\log n)^3 + 2 \log n + c' c_3 (\log n)^3 \log \log n}\rfloor})^{\lfloor F_1(n)\rfloor}
\subseteq Y^{\lfloor F_2(n)\rfloor -1}\end{aligned}\]
for $n$ larger than a constant.
If $Y\subseteq \Alt(\lbrack n\rbrack)$, then $Y^{\lfloor F_2(n)\rfloor
-1} = \Alt(\lbrack n
\rbrack)$. If
$Y$ contains an odd permutation then  
$Y^{\lfloor F_2(n)\rfloor} = \Sym(\lbrack n\rbrack)$. 
\end{proof}

We finally turn to the proof of Proposition~\ref{growth}. 
\begin{proof}[Proof of Proposition~\ref{growth}]
We can assume that $n$ is large enough that $m\geq (\log n)^2 >C(0.9)$,
where $C(0.9)$ is as 
in Lemma \ref{lem:brioche}. Apply Lemma \ref{lem:brioche}
with $d = 0.9$ and $\Sigma = \{\alpha_1,\dotsc,\alpha_m\}$. We obtain 
a set $\Delta\subseteq \Sigma$ such that $|\Delta| \geq 0.9 |\Sigma|$
and $\left(\left(A^{16 m^6}\right)_{\Sigma}\right)_{(\Sigma\setminus \Delta)}|_\Delta$
contains $\Alt(\Delta)$. Let 
\[B^+ = \left\{g\in \left(\left(A^{16 m^6}\right)_{\Sigma}\right)_{(\Sigma\setminus \Delta)} : g|_\Delta \in \Alt(\Delta)\right\},\;\;\; B^- = \left((B^+)^3\right)_{(\Delta)} .\]
This is our initial {\em setup}: we have a large set $B^+$ in the setwise stabilizer
$G_\Sigma$; furthermore, we have constructed a large subset $\Delta\subseteq \Sigma$
such that $B^+\subseteq (G_\Sigma)_{(\Sigma\setminus \Delta)}$ and
$B^+|_\Delta = \Alt(\Delta)$. We also have a set $B^-$ in the pointwise 
stabilizer $G_{(\Sigma)}$.
By (\ref{eq:chrosh}) with $i=m+1$, $\left|\alpha_{m+1}^{B^-}\right| \geq 
\frac{9}{10} n$, and so $\langle B^-\rangle$ has an orbit $\Gamma$ of
length at least $0.9 n$. By Lemma \ref{lem:sonofschreier}, $\Gamma$
is also an orbit of $\langle B^+\rangle$.

We would like $\langle B^-\rangle$ to act as an alternating or symmetric group
on $\Gamma$; let us show that, if this is not the case, we obtain 
{\em descent}.
We are assuming that Proposition \ref{main theorem} holds for $n'<n$ (inductive
hypothesis). Hence, if $\langle B^-\rangle$ has no composition factor
$\Alt(k)$ with $k> 0.95n$, 
then Lemma~\ref{lem:cases3a} ({\em descent}) gives us
\[A^{\lfloor 16 m^6 e^{c_1 (\log n)^3} \cdot F_2(0.95 n)
\rfloor} \supseteq \Alt(\lbrack n\rbrack),
\]
for $n$ larger than an absolute constant, 
where $c_1=c(0.9)$ is from Lemma~\ref{lem:cases3a}. By (\ref{eq:peterson}),
we conclude that (\ref{eq:jlayt}) holds and we are done. (We are assuming
that
$n$ is larger than a constant, so that $16 n^6\leq e^{c_3 \log n}$,
where $c_3>0$ will be set later.)


Thus, we can
suppose from now on that $\langle B^-\rangle$ does have a composition factor
$\Alt(k)$ for some $k> 0.95n$. The only orbit of $\langle B^-\rangle$ that
can be of
length at least $k$ is $\Gamma$, so $\langle B^-\rangle|_\Gamma = 
\langle B^-|_\Gamma\rangle$ must contain $\Alt(k)$ as a section. Hence,
by Lemma \ref{largealt}, $\langle B^-|_\Gamma\rangle \geq \Alt(\Gamma)$.
(We can assume $0.95 n > 84$, and thus Lemma~\ref{largealt} does apply.) Note
we also get that $|\Gamma|> 0.95n$.

Now that we know that $\langle B^-|_\Gamma\rangle \geq \Alt(\Gamma)$,
Corollary \ref{cor0.5} gives us a small set of elements
$Y=\{y_1, y_2, \dotsc , y_6\}\subseteq (B^-)^{\lfloor n^{28 \log n}\rfloor}$ such that 
$\langle Y\rangle|_\Gamma$
is $2$-transitive on $\Gamma$. We apply Lemma \ref{lem:tomev}
({\em creation}) with $H^- = \langle B^-\rangle$, $H^+ = \langle B^+\rangle$,
$B = B^+$ and $r=6$. (The condition $H^- \lhd H^+$ is fulfilled thanks to Lemma 
\ref{lem:sonofschreier}.)

If conclusion (\ref{it:gara}) in Lemma \ref{lem:tomev} holds, then
there is a $b\in B^+ (B^+)^{-1}\setminus \{e\}$ with $\supp(b) \leq 0.05 n$.
Thm.~\ref{bbssmallsupport} thus gives us that
$(A \cup \{b\})^{K n^8} \supseteq \Alt([n])$, where $K=K(0.1)\geq K(0.05)$ is an
absolute constant. (We set $K = K(0.1)$, instead of $K = K(0.05)$,
 because we are planning to use the
same constant later.)
By (\ref{eq:peterson}),
\[2\cdot 48 m^6\cdot K n^8 < 96 K n^{14} \leq F_1(n),\]
and so (provided that $n$ is larger than a constant)
 (\ref{eq:jlayt}) holds and we are done. (This is what we
call an {\em exit} from the procedure.)

We can thus assume that conclusion (\ref{it:garb}) in Lemma \ref{lem:tomev}
holds, i.e., we have {\em created} a set $W = (B^+)^{-1} Y B^+\cap 
\langle B^-\rangle$ with
$|W|\geq |B^+|^{1/6}$. Note that $(B^+)^{-1} Y B^+\subset A^{\lfloor n^{29 \log n}\rfloor}$
(for $n$ larger than a constant) and $|B^+|\geq |\Alt(\Delta)| = 
(1/2) |\Delta|! \geq m^{0.899 m}$ (for $m$ larger than a constant; recall
that $|\Delta|\geq 0.9m$). Hence
\begin{equation}\label{eq:bron}
\left|A^{\lfloor n^{29 \log n}\rfloor} \cap \langle B^-\rangle\right| \geq
m^{0.149 m}.\end{equation}

Now that we have {\em created} 
many elements in the pointwise stabilizer of $\Sigma$, it
is our task to {\em organise} them: we wish
to produce  $\alpha_{m+2},\dotsc,\alpha_{m+\ell+1}$ satisfying (\ref{eq:doublestar}).

This can be done in two ways. One is short and simple, gives a bound
of $l\gg m (\log m)/(\log n)^2$, and results in a bound of $O((\log n)^5 (\log \log n))$ in
the exponent of the final
result. The other is longer, but gives the stronger bound of
$l\gg m (\log m)/(\log n)$ promised in the statement of the proposition,
and results in a bound of $O((\log n)^4 \log \log n)$ in the exponent
of the final result. Let us go through both arguments for the sake of clarity. 

In the first argument, we simply apply Corollary \ref{cor:glub} 
with $\Sym(\Gamma)$ instead of $\Sym(\lbrack n\rbrack)$ and 
$A^{\lbrack n^{29 \log n}\rbrack}\cap \langle B^-\rangle \supset
B^-$ instead of $A$. We obtain that any maximal sequence of elements
$\alpha_{m+2},\dotsc,\alpha_{m+\ell+1}$ satisfying (\ref{eq:doublestar})
must be of length \[\gg (\log |
A^{\lbrack n^{29 \log n}\rbrack}\cap \langle B^-\rangle|)/(\log n)^2
\gg \frac{\log m^{0.149 m}}{(\log n)^2}
\gg \frac{m (\log m)}{(\log n)^2}.\]
Thus $\ell \gg m (\log m)/(\log n)^2$.

Let us now carry out the second argument in detail. The basic idea is that
the creation step has given us enough elements that we can apply
the organiser step several times in succession.

For $i \ge 0$, we define recursively $A_i,B_i \subseteq \langle A \rangle$ and a sequence $\Sigma_i$ of points in $[n]$. Let 
$A_0 = A^{\lfloor n^{29 \log n}\rfloor}$, $m_0=m$, $\Sigma_0 =(\alpha_1,\ldots, \alpha_{m_0+1})$, and $B_0=(A_0)_{(\Sigma_0 \setminus \{\alpha_{m_0 +1} \})}$.

If $A_i,\Sigma_i,B_i$ are already defined then let 
$A'_{i+1}=A_i^{\lfloor 9n^{6} \log n\rfloor}$ and let $\Sigma_{i+1}$ be a 
maximal extension $\Sigma_{i+1}=(\alpha_1,\ldots, \alpha_{m_{i+1}+1})$ of
$\Sigma_{i}=(\alpha_1,\ldots, \alpha_{m_{i}+1})$ such that 
\begin{equation}\label{eq:strid}
\left|\alpha_j^{(A'_{i+1})_{(\alpha_{1},\ldots,\alpha_{j-1})}}\right| \geq 
0.9 n,\end{equation}
for all $j=1,2,\ldots,m_{i+1}+1$. Finally, let 
\[A_{i+1}=(A'_{i+1})^{29 n^6}\;\;\;\; \text{and} \;\;\;\; B_{i+1}=(A_{i+1})_{(\Sigma_{i+1} \setminus \{ \alpha_{m_{i+1}+1} \})}.\] Note that for all $i \ge 0$, $\langle B_i \rangle$ has an orbit $\Gamma_i$ of length at least $0.9n$ because $\left| \alpha_{m_{i}+1}^{B_i} \right| \ge 0.9n$. (We went up to $i=m+1$ in condition
(\ref{eq:chrosh}) and up to $i=m+l+1$ in conclusion (\ref{eq:doublestar})
(rather than $i=m$ and $i=m+l$, respectively) so that we could do this
useful trick!)

We stop the recursion, and set $w:=i$ for the last $i$ for which $A_i$ is defined, if either
\begin{itemize}
\item[$(a)$] $\left| B_i|_{\Gamma_i} \right| < |B_i|$, i.e., there are two elements $b_1,b_2 \in B_i$ such that $b_1b_2^{-1}$ fixes $\Gamma_i$ pointwise; or 
\item[$(b)$] $|\Gamma_i| \le 0.95n$ or $\langle B_i|_{\Gamma_i} \rangle
\not\supset \Alt(\Gamma_i)$  or 
\item[$(c)$] $n^{m_i-m_0} >\sqrt{ m^{0.149m} }$.
\end{itemize}
By \eqref{eq:bron}, we have $|B_0| \ge m^{0.149m}$.

First, we estimate the differences $m_{i+1}-m_i$. If the recursion did not stop after the definition of $A_i,B_i$, and $\Sigma_i$ then, in particular,
 the stopping criterion $(c)$ is not fulfilled at step $i$.
Lemma~\ref{lem:duffy}, applied with $\langle B_0 \rangle$ as $G$, 
$G_{(\Sigma_i \setminus \{\alpha_{m_i}+1\})}$ as $H$, and $B_0$ as $A$, then implies that 
\[ 
|B_i| \ge |B_0^2 \cap H | \ge \frac{|B_0|}{n^{m_i-m_0}} \ge \sqrt{ m^{0.149m} }.
\]
Also, by the criteria $(a)$ and $(b)$, we have $\left| B_i|_{\Gamma_i} \right| = 
|B_i|$ and $\langle B_i\rangle|_{\Gamma_i}$ acts as $\Alt(\Gamma_i)$ or $\Sym(\Gamma_i)$
 on $\Gamma_i$, where 
$|\Gamma_i| >0.95n$.

Since $0.9 n < 0.95 \cdot 0.95 n \leq 0.95 |\Gamma_i|$, 
we can apply Corollary \ref{cor:glub}  with 
$\rho = 0.05$, $B_i|_{\Gamma_i}$ instead of $A$, 
 and $\Gamma_i$ instead of $\lbrack n\rbrack$, and obtain that, for
$1\leq i< w$,
\begin{equation}
\label{eq:mi}
m_{i+1}-m_i > \frac{\log |B_i|}{60 (\log n)^2} \geq 
\frac{c_2 m \log m}{60 (\log n)^2},
\end{equation}
where we define $c_2:= 0.149/2 = 0.0745$. (This is what we have called an 
{\em organiser} step. It is ultimately based on the splitting lemma
(Prop.~\ref{prop:babai}), of which
 Cor.~\ref{cor:glub} is a corollary.)

At the same time, $n^{m_{w-1}-m_0} \leq \sqrt{m^{0.149 m}}$ implies
\[m_{w-1}-m_0  \leq \frac{c_2 \log m}{\log n} m.\]
Since $m_{w-1}-m_0 = \sum_{i=1}^{w-1} (m_i-m_{i-1})$, from \eqref{eq:mi} it follows that 
$$\frac{c_2 \log m}{\log n} m > (w-1)\frac{c_2 m \log m}{60 (\log n)^2}$$
and we conclude that $w -1 < 60 \log n$.
Hence
\[A_w = A_0^{\lfloor 9 n^6 \log n\rfloor^{w}(48n^6)^w} \subseteq A^{\lfloor n^{29 \log n}\rfloor
 \cdot \lfloor 432 n^{12} \log n\rfloor^{w}
} \subseteq
A^{\lfloor n^{c_3 \log n}\rfloor}\]
for $c_3 := 750 > 29+12\cdot 60$,
provided that $n$ is larger than an absolute constant.

If $n^{m_w-m_0} > \sqrt{m^{0.149 m}}$ (stopping condition $(c)$), then
\[ m_w-m_0 \geq \frac{c_2 \log m}{\log n} m,\]
and so, setting $\ell = m_w-m_0$, we obtain (\ref{eq:doublestar}).

(In other words: as long as our {\em organizing} has consumed less than the
square-root of the material we {\em created}, we are organizing rapidly;
if our organizing has consumed at least the square-root of the said material,
then we have already organized plenty.)

If we stopped because condition $(a)$ holds then $A_w^2$ contains a non-trivial element $b_1b_2^{-1}$ with support less than $0.1n$.
By Theorem~\ref{bbssmallsupport}, $(A \cup \{b_1b_2^{-1}\})^{K n^8} \supseteq \Alt([n])$, where $K=K(0.1)$ is an absolute constant.
By (\ref{eq:peterson}),
\[2\cdot \lfloor n^{c_3 \log n}\rfloor \cdot K n^8 \leq F_1(n),\]
and
so we obtain \eqref{eq:jlayt}. (This is an {\em exit} case.)

Finally, suppose we stopped in case $(b)$, i.e., $\langle B_w|_{\Gamma_w}
\rangle \not\supset \Alt(\Gamma_w)$ 
or $|\Gamma_w| \le 0.95n$. 
As $|\Sigma_w| \ge m >C(0.9)$, we can apply Lemma~\ref{lem:brioche} with 
$\Sigma_w\setminus \alpha_{m_w+1}$ as $\Sigma$ and 
$A'_w$ as $A$, to obtain
$\Delta_w \subseteq \Sigma_w\setminus \alpha_{m_w+1}$, 
$|\Delta_w| \ge 0.9 |\Sigma_w\setminus \alpha_{m_w+1}|$
such that 
\[B^+_w = 
(((A'_w)^{16n^6})_{\Sigma_w \setminus \{\alpha_{m_w+1}\}})_{
(\Sigma_w \setminus (\{\alpha_{m_w+1}\} \cup \Delta_w))}\] 
satisfies $(B^+_w)|_{\Delta_w}=\Alt(\Delta_w)$.
(This is a fresh {\em setup}.)
 Also, by Lemma~\ref{lem:sonofschreier}, $B^-_w = \left( (B^+_w) ^3\right)_{(\Delta_w)}$ generates
$\langle B^+_w \rangle_{(\Delta_w)} \lhd \langle B^+_w \rangle$. Note that 
$B^-_w \subseteq B_w$ and $\langle B^-_w \rangle$ has an orbit of length at least $0.9n$, simply because $B^-_w$ contains 
$(A'_w)_{\Sigma_w\setminus \{\alpha_{m_w+1}\}}$, and the orbit of $\alpha_{m_w+1}$ under 
$(A'_w)_{\Sigma_w\setminus \{\alpha_{m_w+1}\}}$ is of length $\geq 0.9 n$ by
(\ref{eq:strid}).

We are ready for another {\em descent}.
The group $\langle B^-_w \rangle$ has no composition factor $\Alt(k)$ with 
$k >0.95 n$, because such a factor would be a section of $\langle B_w
\rangle$ and Lemma~\ref{largealt} would imply that $\langle B_w|_{\Gamma_w}
\rangle$ is an alternating group on $> 0.95 n$ elements, in contradiction with condition $(b)$. Thus the hypotheses of 
Lemma \ref{lem:cases3a} are satisfied with $\delta=0.95$ and $\rho=0.9$ and,
by the assumption that Prop. \ref{main theorem} holds for $n' \leq 0.95 n < n$
(inductive hypothesis),
Lemma \ref{lem:cases3a} gives us that
\[A^{\lfloor n^{c_3 \log n} 
e^{c (\log n)^3} \cdot F_2(0.95 n)
\rfloor} \supseteq \Alt(\lbrack n\rbrack),
\]
where $c=c(0.9)$. We apply (\ref{eq:peterson}), and
conclude that (\ref{eq:jlayt}) holds.

\end{proof}

We now use Proposition~\ref{main theorem} to prove both the Main Theorem 
and Cor.~\ref{directed main} (for $\Sym(n)$ and $\Alt(n)$).

\begin{thm} \label{explicit main}
Let $G = \Sym(n)$ or $\Alt(n)$. 
Then 
\begin{equation}\label{eq:czerwony}\begin{aligned}
\diam(G) &= O(e^{c (\log n)^4 \log\log n}),\\
\overrightarrow{\diam}(G) &= O(e^{(c+1) (\log n)^4 \log\log n}),
\end{aligned}\end{equation}
for an absolute constant $c>0$.
\end{thm}

As we shall see, $c_1 = 49071$ is valid (and by no means
optimal).
\begin{proof}
We must find functions $F_1$, $F_2$ satisfying (\ref{eq:peterson}) and
(\ref{eq:worsi}). We can set 
\[F_2(n) = e^{(\log n)^3 + 2 \log n + c' c_3 (\log n)^3 \log \log n}
F_1(n) + 2\]
for $c'>c_2$ arbitrary. Now we must make sure that
\begin{equation}\label{eq:grenou}\begin{aligned}
F_1(n)\geq\;  &n^{c_3 \log n} e^{c_1 (\log n)^3}\\
&\cdot
\left(e^{c' c_3 (\log 0.95 n)^3 \log \log 0.95 n + 
(\log 0.95n)^3 + 2 \log 0.95 n} F_1(0.95n) + 2\right)
.\end{aligned}\end{equation}
(Here we can assume $n>1$, so that $\log \log n$ is well-defined.) Choose $c_4 > c' c_3$. Then,
for $n$ larger than a constant $n_0'$ depending only on $c_1$, $c_3$,
$c'$ and $c_4$, (\ref{eq:grenou}) will hold
 provided that
\begin{equation}\label{eq:maga}
F_1(n) \geq e^{c_4 (\log n)^3 \log \log n} \max(F_1(0.95 n),1).\end{equation}
For any $c> c_4/(4 |\log 0.95|)$ and any $C\geq 1$, (\ref{eq:maga}) is fulfilled by
\[F_1(n) = C e^{c (\log n)^4 \log \log n}
,\]
provided that $n$ is larger than a constant $n_0''$ depending only on
$c$ and $c_4$. We set $C = n_0'''!$, where $n_0'''=\max(n_0,n_0',n_0'', 2 K)$. 
Then (\ref{eq:peterson}) holds for all $n\geq n_0'''$, and 
(\ref{eq:worsi}) holds with $n_0!$ replaced by $n_0'''!$.
We now apply Proposition \ref{main theorem} for our $n$,
with $n_0$ replaced by $n_0'''$; it uses Proposition \ref{growth},
which in turn uses Proposition \ref{main theorem} for smaller $n$, and so
on. The recursion ends when $n\leq \max(n_0''',1)$, as then Proposition
\ref{main theorem} is trivially true (due to the bound $F_2(n)\geq n_0'''!$ in
(\ref{eq:worsi})).

We obtain that
\begin{equation}\label{eq:zulu}
\diam(\Gamma(G,Y)) \leq C e^{c (\log n)^4 \log \log n}\end{equation}
for any set $Y$ of generators of $G$ with $Y = Y^{-1}$, $e\in Y$.
A quick calculation shows that, since $c_2 =0.0745$ and $c_3 = 750$
(see the proof of Prop.~\ref{growth}), we can set
$c' = 13.423>1/0.0745$, $c_4 = 10068>c' c_3$ and
\[c = \left\lfloor \frac{c_4}{4 |\log 0.95|}\right\rfloor =  49071.\]


Let $A$ be an arbitrary set of generators of $G$.
Let $Y = A \cup A^{-1} \cup \{e\}$. The undirected Cayley graph
$\Gamma(G,Y)$ is just the undirected Cayley graph
$\Gamma(G,A)$ with a loop at every vertex; their diameters are the same.
Thus, by (\ref{eq:zulu}),
\[\diam(\Gamma(G,A)) = \diam(\Gamma(G,Y)) \leq C e^{c (\log n)^4 \log\log n}.\]
By \cite[Cor. 2.3]{MR2368881},
\[\diam(\vec{\Gamma}(G,A)) \le O\left(\diam(G) (n \log n)^2\right)
\le O\left(e^{(c+1) (\log n)^4 \log\log n}\right).\]
\end{proof}

\bibliographystyle{alpha}
\bibliography{sym.v4}
\end{document}